\documentclass[11pt,english]{smfart}
\usepackage[T1]{fontenc}
\usepackage[english]{babel}
\usepackage{amssymb}
\usepackage{amsmath}
\numberwithin{equation}{section}
\theoremstyle{plain}
\theoremstyle{definition}
\newtheorem{definition}{Definition}[section]
\newtheorem{theorem}[definition]{Theorem}

\newtheorem{lemma}[definition]{Lemma}
\newtheorem{proposition}[definition]{Proposition}

\theoremstyle{remark}
\newtheorem*{remark}{Remark}
\theoremstyle{remark}

\theoremstyle{definition}

\begin{document}
\def\smfbyname{}

\title[Polygonal numbers]{Sums of four generalized polygonal numbers of almost prime length}
%\date {}
\author{Kwan to Ng}
\address{Mathematics Department, The University of Hong Kong, Pokfulam, Hong Kong}
\email{u3011714@connect.hku.hk}

%\address{Institut Henri Poincar\'e\\
%11 rue Pierre et Marie Curie, F-75231 Paris cedex 05}
%\email{christia@dma.ens.fr}
%\urladdr{http://smf.emath.fr/}
%\keywords{\LaTeXe, SMF, typesetting}

\begin{abstract}
In this paper, we consider sums of four generalized $m-$gonal numbers whose parameters are restricted to integers with a bounded number of prime divisors. With some restriction on $m$ modulo 30, we show that for $n$ sufficiently large, it can be represented as such a sum, where the parameters are restricted to have at most 988 prime factors.
\end{abstract}

\maketitle

%\tableofcontents
\section{Introduction}
For $m \in \mathbb{N}_{\geq3}$ and $x \in \mathbb{Z}$, set \medskip $$p_m(x):= \frac{(m-2)x^2}{2}-\frac{(m-4)x}{2}.$$ \\
If $x \in \mathbb{N},$ then the number $p_m(x)$ counts the number of dots in a regular $m$-gon of length $x.$ In general, for $x \in \mathbb{Z},$ we call $p_m(x)$ the generalized $m-$ gonal numbers.
\\For $n\in \mathbb{N}, \alpha_j\in \mathbb{N}, $ consider the equation
\begin{equation}\label{eqn:summgonal}
    \alpha_1p_m(x_1)+\alpha_2p_m(x_2)+\alpha_3p_m(x_3)+\alpha_4p_m(x_4)=n.
\end{equation}
It is obvious that if $(\alpha_1,\alpha_2,\alpha_3,\alpha_4)$ is too arbitrary, i.e.  $(\alpha_1,\alpha_2,\alpha_3,\alpha_4)=(p,p,p,p),$ then $n$ that can be represented is restricted. We hence restrict ourselves to the case where $\prod_{j=1}^4\alpha_j$ is odd and square-free.
To understand the above sum and the relationship with shifted lattice theory, we complete the square so that (\ref{eqn:summgonal}) becomes
$$\sum_{j=1}^{4}\alpha_{j}X_{j}^{2}=8(m-2)n+\sum_{j=1}^{4}\alpha_{j}(m-4)^{2},$$
        where $X_j=(2(m-2)x_{j}+4-m).$
        Taking $L=2(m-2)\mathbb{Z}^4$ and $v=(4-m,...,4-m) \in \mathbb{Z}^4,$ above is related to counting points of a given Euclidean distance on the shifted lattice $X=L+v.$
\begin{theorem}\label{thm:main}
    Suppose that $m$ is odd, $m-4\not\equiv 0 \pmod3$ and $m-4\not\equiv 0 \pmod5.$
    Further suppose $\text{ord}_p(\prod_{j=1}^4\alpha_j)\leq 1$ for any odd prime $p.$ Then for $n$ sufficiently large, the equation (\ref{eqn:summgonal}) is solvable with each $x_j$ containing at most $988$ prime factors.
\end{theorem}
\begin{remark}
    We further assume that $m-4\not\equiv 0 \pmod3$ and $m-4\not\equiv 0 \pmod5$ due to the following technical issue.
    Suppose we have $m-4\equiv 0 \pmod3$ or $m-4\equiv 0 \pmod5,$ then $X=2(m-2)x+4-m \not\equiv 0 \pmod p$ is equivalent to $x\not\equiv 0 \pmod p.$ Therefore the integers n that can be represented as (\ref{eqn:summgonal}) with $x_1,x_2,x_3,x_4\not\equiv 0 \pmod p$ is restricted: when $p = 3,$ only those $n$ with $n \equiv 0 \pmod p$, when $p = 5$ only those $n$
with $n \not\equiv 0 \pmod p.$
\end{remark}
The paper is organized as follows. In Section 2, we recall the theory of quadratic lattices and shifted lattices and the interrelation between modular forms and theta series. In Section 3, we compute local densities for the Eisenstein series component of the theta series and obtain lower bounds for the Fourier coefficients of the Eisenstein series component. In section 4, we bound the Fourier coefficients coming from the contribution of the cuspidal part of the theta function. In Section 5, we obtain bounds to set up the sieve required in the proof of Theorem \ref{thm:main}. We apply sieving techniques and prove Theorem \ref{thm:main} in Section 6.

\section{Preliminaries}
\Subsection{Setup and notation}
Throughout this paper, we suppose $m$ to be odd. Let $V$ be a positive definite quadratic space over $\mathbb{Q}$ of rank $k$ associated with an non-degenerate symmetric bilinear form 
\begin{equation*}
    B: V \times V\rightarrow \mathbb{Q} \text{ with } Q(x)=B(x,x) \text{ for any }
x\in V.
\end{equation*}
A quadratic $\mathbb{Z}$-lattice is a free $\mathbb{Z}$-module 
\begin{equation*}
    L=\mathbb{Z}e_1+\mathbb{Z}e_2+\cdot\cdot\cdot+\mathbb{Z}e_k
\end{equation*}
of rank $k,$ where $e_1,...,e_k$ is a basis of $V.$ The dual lattice of $L$ is defined by 
\begin{equation*}
    L^{\#}:=\{v\in V: B(v,x)\in \mathbb{Z} \text{ }\forall x\in L \}.
\end{equation*}
The $k\times k$ symmetric matrix $A_{L}:=(B(v_i,v_j))$ is called the Gram matrix of $L$ with respect to a basis $v_1,...,v_k,$ and we write $L\cong (B(v_i,v_j)).$ We generally choose the vectors $v_i$ to be the standard basis element $e_i.$ The discriminant $d_L$ is defined to be the determinant of $(B(e_i,e_j)).$ If the Gram matrix is diagonal, then we write $L\cong \langle Q(e_1),...,Q(e_k) \rangle.$ For a rational prime $p,$ we define the localization of $V$ and $L$ by 
\begin{equation*}
    V_p:=V\otimes_{\mathbb{Q}}\mathbb{Q}_p \text{ and } L_p:= L\otimes_{\mathbb{Z}}\mathbb{Z}_p.
\end{equation*}
\\A lattice coset is a coset $L+v\in V/L$ where $v\in V.$ The conductor of the coset $L+v$ is the smallest integer $a$ such that $av\in L.$
Now consider 
$$f_{\boldsymbol{\alpha}}({\textbf{x}}):=\sum_{j=1}^4\alpha_jp_m(x_j)$$
and write $$h=h(n):=8(m-2)n+\sum_{j=1}^4\alpha_j(m-4)^2.$$
Then $n$ is represented by $f$ iff $h$ is represented by $X^{\textbf{1}}=L^{\textbf{1}}+v$ where 
\begin{multline*}
L^{\textbf{1}}=\mathbb{Z}(e_1)+\mathbb{Z}(e_2)+\mathbb{Z}(e_3)+\mathbb{Z}(e_4) 
\\\cong \langle 4(m-2)^2\alpha_1, 4(m-2)^2\alpha_2, 4(m-2)^2\alpha_3, 4(m-2)^2\alpha_4\rangle
\end{multline*}
and $v=\frac{4-m}{2(m-2)}(e_1+e_2+e_3+e_4).$
More generally, for ${\textbf{d}}=(d_1,d_2,d_3,d_4)\in \mathbb{N}^4,$ define the coset $X^{\textbf{d}}:=L^{\textbf{d}}+v,$ where $L^{\textbf{d}}$ is the sublattice of $L^{\textbf{1}}$ defined by 
$$L^{\textbf{d}}=\mathbb{Z}(d_1e_1)+\mathbb{Z}(d_2e_2)+\mathbb{Z}(d_3e_3)+\mathbb{Z}(d_4e_4).$$
Then $n$ is represented by $f$ with the restriction $d_j\mid x_j$ iff $h$ is represented by $X^{\textbf{d}}.$ Therefore, we are interested in the number of solution to $$h=\sum_{j=1}^4(2(m-2)d_jx_j+4-m)^2.$$
\Subsection{Modular forms}
We will introduce modular form of half-integral weight in the following.
\\For positive integer N, we define 
\begin{equation*}
    \Gamma_0(N):=\{\begin{pmatrix}
        a & b \\
        c & d 
    \end{pmatrix}\in \Gamma_0(N): a,d\equiv 1 \pmod N\},
\end{equation*}
\begin{equation*}
    \Gamma_1(N):=\{\begin{pmatrix}
        a & b \\
        c & d 
    \end{pmatrix}\in \text{SL}_2(\mathbb{Z}): c\equiv 0 \pmod N\},
\end{equation*}
to be two natural congruence subgroups of $\text{SL}_2(\mathbb{Z}).$
\\For $\gamma= \begin{pmatrix}
        a & b \\
        c & d 
    \end{pmatrix}\in \Gamma_0(4)$ and $k\in \frac{1}{2}\mathbb{Z},$ we define the slash operator on a function $f:\mathbb{H}\rightarrow \mathbb{C}$ by 
    $$f\mid_k\gamma(z)=(\frac{c}{d})^{2k}\epsilon_d^{2k}(cz+d)^{-k}f(\gamma z),$$
    where $\epsilon_d= \begin{cases}
			1, & \text{if $d\equiv 1 \pmod4$}\\
            i & \text{if $d\equiv 3 \pmod4$}
		 \end{cases} ,$
$(\frac{\cdot}{\cdot})$ is the Kronecker symbol, and $\gamma$ acts on $\mathbb{H}$ by fractional linear transformations defined by $\gamma z:=\frac{az+b}{cz+d}.$ Then we call $f$ a (holomorphic) modular form of weight $k$ on $\Gamma\subseteq \Gamma_0(4)$ where $\begin{pmatrix}
        1 & 1 \\
        0 & 1 
    \end{pmatrix}\in \Gamma$ with character $\chi$ if 
    \begin{enumerate}
    \item $f\mid_k\gamma=\chi(d)f$ for any $\gamma= \begin{pmatrix}
        a & b \\
        c & d 
    \end{pmatrix}$ in $\Gamma,$
    \item $f$ is holomorphic on $\mathbb{H},$
    \item $f(z)$ grows at most polynomially in $y$ as $z=x+iy \rightarrow \mathbb{Q} \cup \{i\infty\}.$ 
    \end{enumerate}
    Then we call $f$ a cusp form if $f(z)\rightarrow 0$ as $z\rightarrow \mathbb{Q} \cup \{i\infty\}.$ Note that if $f$ is a modular form for a congruence group $\Gamma $ that contain $\begin{pmatrix}
        1 & 1 \\
        0 & 1 
    \end{pmatrix}$, then $f$ has a Fourier series expansion
    $$f(z)=\sum_{n=0}^{\infty}a(n)q^n,$$ where $q=e^{2\pi iz}.$ If $f$ is a cusp form, then $a(0)=0.$
\Subsection{Theta series for cosets}
Let $X=L+v$ be a coset in a quadratic space $V$ of rank $k.$. For a positive integer $n$, we define     
$$R_X(n):=\{x \in X : Q(x) = n\}, r_X(n):=\mid R_X(n)\mid.$$
Then the theta series $\Theta_X(z)$ of the coset $X$ is defined as
$$\Theta_X(z):=\sum_{x\in X}q^{Q(x)}=\sum_{n=0}^{\infty}r_X(n)q^n.$$
Using \cite[Proposition 2.1]{Banerjee}, a straightforward calculation 
shows that the theta series of cosets of rank $k$ are modular forms of weight $\frac{k}{2}.$ Therefore it splits into a sum of Eisenstein series $E_{X}$ and a cusp form $G_{X}.$ Let $a_{E_{X}}(n)$ and $a_{G_{X}}(n)$ denote the $n-$th Fourier coefficient of $E_{X}$ and $G_{X}$ for any integer $n\geq 0,$ respectively. We have 
$$r_{X}(n)=a_{E_{X}}(n)+a_{G_{X}}(n).$$

\section{Eisenstein series components}
Let $X=L+v$ be a lattice coset of rank $k,$ and let $h\in \mathbb{Q}.$ For $p$ prime and $z\in \mathbb{Q}_p,$ we define $\textbf{e}_p(z)=e^{2\pi i(-y)}$ with $y\in \cup_{t=1}^{\infty}p^{-t}\mathbb{Z}$ such that $z-y\in\mathbb{Z}_p.$ Let $\lambda(v)$ and $\lambda_p$ be the characteristic functions of $X$ and $X_p$ respectively. Normalize the measures $dv$ and $d\sigma$ on $L_p$ and $\mathbb{Z}_p$ so that $\int_{L_p} \,dv=\int_{\mathbb{Z}_p} \,dv=1.$
    Define the local density at p by
    \begin{equation}\label{eqn:localden}
    b_p(h,\lambda,0):=\int_{\mathbb{Q}_p}\int_{L_p}\textbf{e}_p(\sigma(Q(v+\nu)-h)\,dv\,d\sigma.
    \end{equation}
    Shimura \cite[Theorem 1.5]{Shimura} showed that the $h$-th Fourier coefficient of the Eisenstein series part of $\Theta_X$ can be expressed as the product of local densities. We refer the readers to \cite{Van der Blij} and \cite{Weil} for more details.
    \\We now provide the statement for the special case when $k = 4.$
    \begin{theorem}
        Let $X=L+v$ be a quaternary lattice coset, and let $h\in \mathbb{Q}.$ Then 
        $$a_{E_{X}}(h)=\frac{(2\pi)^2h}{[\frac{1}{2}L^\#:L]^{\frac{1}{2}}\Gamma(2)L(2,\psi)}\cdot \prod_{p\mid e_1}\frac{b_p(h,\lambda,0)}{(1-\psi (p)p^{-2})}\cdot\prod_{{p\mid h\atop{p \nmid e_1}}}\gamma_p(2).$$
    \end{theorem}
    Here $\psi(\cdot)$ is primitive character associated to the real character $(\frac{-1}{\cdot}),$ $e_1$ is the product of all primes p at which $h\not\in \mathbb{Z}_p,$ $e'$ is the product of all primes such that $L^{\#}\neq 2L_p,$ the numbers $\gamma_p(s)$ are given in  \cite[Section 1.6]{Shimura}.
    Let $\lambda_{\textbf{d}}$ be the characteristic function of $X^{\textbf{d}}.$ In our case $e_1$ is the product of all primes dividing $2(m-2)\prod_{j=1}^4\alpha_jd_j,$ and $e_1=e'.$ Therefore we have 
    \begin{equation}\label{eqn:eiseneeqn}
        a_{E_{X^{\textbf{d}}}}(h)=\frac{(2\pi)^2h}{(16d_{L^{\textbf{d}}})^\frac{1}{2}\Gamma(2)L(2,\psi)}\cdot \prod_{p\mid e_1}\frac{b_p(h,\lambda,0)}{(1-\psi (p)p^{-2})}\cdot\prod_{{p\mid h\atop{p \nmid e_1}}}\gamma_p(2).
    \end{equation}
    Furthermore, $t_p$ defined in \cite[Section 1.6]{Shimura} is equal to 2 in our case. Therefore for any odd prime $p,$ we have
    $$\gamma_p(2)=
    \begin{cases}
        \frac{1-(-p^{1-s})^{v+1}}{1+p^{1-s}}, & \text{if $\psi(p)\neq 0$ and $c_p\in \mathbb{Z}_p^{\times}$,}\\
        \frac{1-(-p^{1-s})^{v+2}+p^{2-s}(1-(-p^{1-s})^v)}{(1+p^{-s})(1+p^{1-s})}, & \text{if $\psi(p)\neq 0$ and $c_p\not\in \mathbb{Z}_p^{\times}$},\\
        1+\xi_p(c_ph)p^{(v+\sigma)(1-s)}, & \text{if $\psi(p)=0$}.
    \end{cases}$$
    Where the numbers $v, c_p, \sigma$ and the function $\xi_p$ are defined in \cite[Section 1.6]{Shimura}.
    \subsection{Computation of local densities of $X^\textbf{d}$.}
     In this subsection, we compute the local densities $b_p(h, \lambda_\textbf{d}, 0)$ of $X^\textbf{d}$ in order to evaluate $a_{E_{X^{\textbf{d}}}}(h)$ using (\ref{eqn:eiseneeqn}). We first require a formula for 
     $$\tau_{\alpha_j,p,s}(\sigma):= \int_{\mathbb{Z_{p}}}e_{p}(4(m-2)s\sigma\alpha_j((m-2)sx^2-(m-4)x))\,dx$$
     for $\sigma \in \mathbb{Q}_{p}.$ Then by \cite[Lemma 3.2]{Banerjee}, we have 
     \begin{lemma}\label{tau}
         If $\alpha_js\in \mathbb{Z}_{p}$ satisfies $p\mid 2(m-2)\alpha_js,$ then we have
         $$ \tau_{\alpha_j,p,s}(\sigma)=
        \begin{cases}
        1, & \text{if $4(m-2)\alpha_js\sigma \in \mathbb{Z}_p$,}\\
        0, & \text{otherwise.}
     \end{cases}$$
         
     \end{lemma}
     We first compute the local density for the case $p=2.$

    \begin{lemma}\label{lem:2den}
    The local density if $X^{\textbf{d}}$ at the prime 2 is given by
    $$ b_2(h,\lambda_{\textbf{d}},0)=
        \begin{cases}
        2^{2+\text{min}_j\{\text{ord}_2(\alpha_jd_j)\}}, & \text{if $8(m-2)n\in 4\gcd(\alpha_1d_1,\alpha_2d_2,\alpha_3d_3,\alpha_4d_4)\mathbb{Z}_2$,}\\
        0, & \text{otherwise.}
     \end{cases}$$
    \begin{proof}
        Note that we have 
        \begin{equation}\label{eqn:quadform}
            Q(u+\nu)= \sum_{j=1}^{4}4\alpha_j(m-2)^2(d_jx_j-\frac{m-4}{2(m-2)})^2,
        \end{equation}
        where $\nu=x_1(d_1e_1)+x_2(d_2e_2)+x_3(d_3e_3)+x_4(d_4e_4).$ Then by (\ref{eqn:localden}) and Lemma \ref{tau}, we have
        \begin{align*}
            b_2(h,\lambda_{\textbf{d}},0)&= \int_{\mathbb{Q}_p}e_p(\sigma(\sum_{j=1}^4\alpha_j(m-4)^2-h))\prod_{j=1}^4\tau_{\alpha_j,2,d_j}(\sigma)\,d\sigma
            \\&=\int_{\frac{1}{4\gcd(\alpha_1d_1,\alpha_2d_2,\alpha_3d_3,\alpha_4d_4)}\mathbb{Z}_2}e_p(-8n(m-2)\sigma)\,d\sigma
            \\&=\begin{cases}
        2^{2+\text{min}_j\{\text{ord}_2(\alpha_jd_j)\}}, & \text{if $8(m-2)n\in 4\gcd(\alpha_1d_1,\alpha_2d_2,\alpha_3d_3,\alpha_4d_4)\mathbb{Z}_2$,}\\
        0, & \text{otherwise.}
     \end{cases}
        \end{align*}
        
    \end{proof}
    \end{lemma}

    We now compute the local density for odd $p\mid(m-2).$
        \begin{lemma}\label{lem:oddden}
    The local density if $X^{\textbf{d}}$ at odd prime $p\mid(m-2)$ is given by
    $$ b_p(h,\lambda_{\textbf{d}},0)=
        \begin{cases}
        p^{\text{ord}_p(m-2)+\text{min}\{\text{ord}_p(\alpha_jd_j)\}}, & \text{if $n\in \gcd(\alpha_1d_1,\alpha_2d_2,\alpha_3d_3,\alpha_4d_4)\mathbb{Z}_p$,}\\
        0, & \text{otherwise.}
     \end{cases}$$
    \begin{proof}
        By (\ref{eqn:localden}), Lemma \ref{tau} and (\ref{eqn:quadform}), we have
        \begin{align*}
            b_p(h,\lambda_{\textbf{d}},0)&= \int_{\mathbb{Q}_p}e_p(\sigma(\sum_{j=1}^4\alpha_j(m-4)^2-h))\prod_{j=1}^4\tau_{\alpha_j,2,d_j}(\sigma)\,d\sigma
            \\&=\int_{\frac{1}{(m-2)\gcd(\alpha_1d_1,\alpha_2d_2,\alpha_3d_3,\alpha_4d_4)}\mathbb{Z}_p}e_p(-8n(m-2)\sigma)\,d\sigma
            \\&=\begin{cases}
        p^{\text{ord}_p(m-2)+\text{min}_j\{\text{ord}_p(\alpha_jd_j)\}}, & \text{if $8(m-2)n\in \gcd(\alpha_1d_1,\alpha_2d_2,\alpha_3d_3,\alpha_4d_4)\mathbb{Z}_p$,}\\
        0, & \text{otherwise.}
     \end{cases}
        \end{align*}
    \end{proof}
    \end{lemma}
    
     For $p\geq 3,$ we will use the formula from \cite[Theorem 4.2]{Kane} to compute the local densities. We provide the statement for $r=4$ in the following.

     \begin{theorem}\label{Thm:localdenexplict}
         Let $p\geq 3$ be an odd prime. Suppose that $n\in \mathbb{Z}_{p}$ and $\phi(\textbf{x})=\sum_{i=1}^4(b_ix_i^2+c_ix_i)$ with $b_i,c_i\in \mathbb{Z}_{p}$ for $1\leq i\leq 4.$ For $1\leq i\leq 4,$ we define $t_i=\text{min}(\text{ord}_p(b_i),\text{ord}_p(c_i)).$ We set 
         \begin{align*}
             & D_p:=\{1\leq i \leq4\mid \text{ord}_p(b_i)>\text{ord}_p(c_i)\},
             \\& N_p:=\{1\leq i \leq4\mid \text{ord}_p(b_i)\leq\text{ord}_p(c_i)\}.
         \end{align*}
         and set $\textbf{t}_{d}:=\text{min}\{t_i\mid i\in D_p\}.$ We further define
         $$\textbf{n}:=n+\sum_{i\in N_p}\frac{c_i^2}{4b_i}.$$
         If $\textbf{n}\neq 0,$ we assume that $\textbf{n}=\textbf{u}_np^{\textbf{t}_n}$ with $\textbf{u}\in \mathbb{Z}_p^{\times}$ and $\textbf{t}_n\in \mathbb{Z}.$ Otherwise we set $\textbf{t}_n:=\infty.$ For an integer $t\in \mathbb{Z},$ we define
         $$L_p(t):=\{i\in N_p\mid t_i-t<0 \text{ and odd}\}, l_p(t):=\mid L_p(t)\mid.$$
         Then we have 
         $$b_p(h,\lambda_{\textbf{d}},0)=1+(1-\frac{1}{p})\sum_{1\leq t \leq \text{min}(\textbf{t}_d,\textbf{t}_n)\atop l_p(t) \text{ even}}\delta_p(t)p^{\tau_p(t)}+\delta_p(\textbf{t}_n+1)w_pp^{\tau_p(\textbf{t}_n+1)},$$
         where for any integer $t\in \mathbb{Z},$ we define
         $$\delta_{p}(t):=\epsilon_p^{3l_p(t)}\prod_{i\in L_p(t)}(\frac{u_i}{p}), \tau_p(t):=t+\sum_{i\in N_p\atop t_i<t}\frac{t_i-t}{2},$$
         with $u_i:=p^{-\text{ord}(b_i)}b_i\in\mathbb{Z}_p^{\times}$ and define
         $$
        w_p:=\begin{cases}
			0, & \text{if $\textbf{t}_n\geq\textbf{t}_d$,}\\
            -\frac{1}{p}, & \text{if $\textbf{t}_n<\textbf{t}_d$ and $l_p(\textbf{t}_n+1)$ is even,}
            \\
            \epsilon_p(\frac{\textbf{u}_n}{p})\frac{1}{\sqrt{p}}, & \text{if $\textbf{t}_n<\textbf{t}_d$ and $l_p(\textbf{t}_n+1)$ is odd.}
		 \end{cases}
        $$
        
     \end{theorem}
     
     Using Theorem \ref{Thm:localdenexplict}, we have the following result.
     \begin{lemma}
        Let $p$ be an odd prime such that $p\nmid (m-2)(m-4), p\mid \prod_{j=1}^4d_j,$ and $p\nmid \prod_{j=1}^4\alpha_j.$
        \\When $\mid N_p\mid=0,$ we have
        $$b_p(h,\lambda_{\textbf{d}},0)=
        \begin{cases}
            p & \text{if $N\in p\mathbb{Z}_p$,}
            \\0 & \text{if $N\not\in p\mathbb{Z}_p$.}
        \end{cases}$$
        \\When $\mid N_p\mid=1,$ we have
        $$0\leq b_p(h,\lambda_{\textbf{d}},0)\leq 2.$$
        \\When $\mid N_p\mid=2,$ we have
        $$\frac{1}{p}\leq b_p(h,\lambda_{\textbf{d}},0)\leq 2-\frac{1}{p}.$$
        \\When $\mid N_p\mid=3,$ we have
        $$1-\frac{1}{p}\leq b_p(h,\lambda_{\textbf{d}},0)\leq 1+\frac{1}{p}.$$

     \end{lemma}
     \begin{proof}
         For $\mid N_p \mid = 0,$ we have $\textbf{t}_d=1$ and $L_p(t) = \emptyset,$ thus $\delta_p(t)=1$ and $\tau_p(t)=t.$ Therefore \begin{align*}
         b_p(h,\lambda_{\textbf{d}},0)&=1+(1-\frac{1}{p})\sum_{1\leq t \leq \text{min}(\textbf{t}_d,\textbf{t}_n)}p^t+w_p\cdot p^{\textbf{t}_n+1}
         \\&=
         \begin{cases}
             p & \text{if $N\in p\mathbb{Z}_p$,}
            \\0 & \text{if $N\not\in p\mathbb{Z}_p$.}
         \end{cases}
         \end{align*}
         For $\mid N_p \mid =1, $ wlog we assume $1\in N_p.$ Then $t_1=0.$ Now the computation splits into two cases, which is $t_1<\text{min}(\textbf{t}_d,\textbf{t}_n)$ or $t_1=\text{min}(\textbf{t}_d,\textbf{t}_n).$
         For the first case, as we assume $0=t_1< \text{min}(\textbf{t}_d,\textbf{t}_n),$ thus $1=\textbf{t}_d\leq \textbf{t}_n,$ hence $\text{min}(\textbf{t}_d,\textbf{t}_n)=1$ and $w_p=0$. Now as $l_p(1)=1,$ the sum in the middle is empty and the local density is 
         $$b_p(h,\lambda_{\textbf{d}},0)=1$$
         For $t_1=\text{min}(\textbf{t}_d,\textbf{t}_n),$ we have $\textbf{t}_d>\textbf{t}_n=0.$ Therefore the local density is bounded by 
         $$0\leq b_p(h,\lambda_{\textbf{d}},0)=1+(\frac{u_1}{p})(\frac{\textbf{u}_n}{p})\leq 2.$$
         The remaining cases can be compute similarly.
     \end{proof}
    We are now ready to give a lower bound for $a_{E_{X^{\textbf{1}}}(h)}.$
     
     \begin{lemma}
     We have
        $$a_{E_{X^{\textbf{1}}}}(h) \gg_{m,\epsilon} h^{1-\epsilon}.$$ 
        \begin{proof}
           By using (\ref{eqn:eiseneeqn}) to the case $\textbf{d}=1,$ we have 
           \begin{align*}
           a_{E_{X^{\textbf{1}}}}(h)=\frac{(2\pi)^2h}{(16d_{L^{\textbf{1}}})^\frac{1}{2}\Gamma(2)L(2,\psi)}\cdot \prod_{p\mid 2(m-2)\prod_{j=1}^4\alpha_j}\frac{b_p(h,\lambda,0)}{(1-\psi (p)p^{-2})}\cdot\prod_{{p\mid h\atop{p \nmid 2(m-2)\prod_{j=1}^4\alpha_j}}}\gamma_p(2).
           \end{align*}
           By using Lemma \ref{lem:2den} and Lemma \ref{lem:oddden}, we have 
           \begin{align*}
           \prod_{p\mid 2(m-2)\prod_{j=1}^4\alpha_j}\frac{b_p(h,\lambda,0)}{(1-\psi (p)p^{-2})} &=\frac{16(m-2)}{3}\prod_{p\mid (m-2)\prod_{j=1}^4\alpha_j}\frac{p^{\text{min}_j\{\text{ord}_p(\alpha_j)\}}}{1-\psi(p)p^{-2}}
           \\&\geq \frac{16(m-2)}{3}\prod_{p\mid (m-2)\prod_{j=1}^4\alpha_j}\frac{1}{1+p^{-2}}
           \end{align*}
           For $p\mid h$ and $p \nmid 2(m-2)\prod_{j=1}^4\alpha_j,$ by \cite[(3.30)]{Rosser} we have the following bound
           \begin{align*}
               \prod_{{p\mid h\atop{p \nmid 2(m-2)\prod_{j=1}^4\alpha_j}}}\gamma_p(2) \geq \prod_{p\leq h}(1-\frac{1}{p})\gg h^{-\epsilon}
           \end{align*}
           Hence we have
           $$a_{E_{X^{\textbf{1}}}}(h)\gg_{m,\epsilon} h^{1-\epsilon}.$$
        \end{proof}
     \end{lemma}
     
    \subsection{Ratio of Eisenstein series.}
    Let $S_{2^j}$ be the set of $d\in \mathbb{N}$ with either $d$ odd and squarefree or $d=2^jd'$ with $d'$ odd and squarefree. From (\ref{eqn:eiseneeqn}), we have for $\textbf{d}\in S_1^4$ and $\boldsymbol{\ell}\in S_1^4$ with $\gcd(d_1d_2d_3d_4,\ell_1\ell_2\ell_3\ell_4)=1$ that 
    \begin{equation}\label{eqn:eisenratio}
    \begin{split}        &\frac{a_{E_{X^{\boldsymbol{d\ell}}}(h)}}{a_{E_{X^{\boldsymbol{\ell}}}(h)}} 
        \\&=\frac{1}{d_1d_2d_3d_4}\cdot\prod_{p\mid (m-2)}\frac{b_p(h,\lambda_{\boldsymbol{d\ell}},0)}{b_p(h,\lambda_{\boldsymbol{\ell}},0)}\cdot \prod_{p\mid d_1d_2d_3d_4\atop p\nmid 2(m-2)}\frac{b_p(h,\lambda_{\textbf{d}},0)}{1-\psi(p)p^{-2}}\cdot \prod_{p\mid (h,d_1d_2d_3d_4)\atop p\nmid 2(m-2)}\gamma_p^{-1}(2).
     \end{split}
    \end{equation}
    \begin{remark}
        Note that (\ref{eqn:eisenratio}) follows from the fact that for general $\textbf{d}\in \mathbb{N}^4,$ we have $b_p(h,\lambda_{\textbf{d}},0)=b_p(h,\lambda_{p^{\text{ord}_p\textbf{d}}},0).$
    \end{remark}
    Define
    \begin{equation}
        \begin{split}
            \beta_{X^{\textbf{p}^{\textbf{c}}},p}(h):=\frac{1}{p^{c_1+c_2+c_3+c_4}}& (\frac{b_p(h,\lambda_{\textbf{p}^c},0)}{b_p(h,\lambda_{\textbf{1}},0)})^ {\delta_{p\mid (m-2)}} 
            \\& \cdot (b_p(h,\lambda_{\textbf{p}^{\textbf{c}}},0)\cdot\frac{1}{1-\psi(p)p^{-2}}\cdot \gamma_p^{-1}(2))^{\delta_{p\nmid 2(m-2),\textbf{c}\neq \textbf{0}}}
        \end{split}
    \end{equation}
    where $p^{\textbf{c}}=(p^{c_1},p^{c_2},p^{c_3},p^{c_4})$ for $\textbf{c}=(c_1,c_2,c_3,c_4).$
    Therefore we have
    \begin{equation}
        \frac{a_{E_{X^{\boldsymbol{d\ell}}}(h)}}{a_{E_{X^{\boldsymbol{\ell}}}(h)}}=\prod_{p \text{ odd}}\beta_{X^{\textbf{p}^{ord_p(\textbf{d})}},p}(h),
    \end{equation}
    where the product runs over all odd prime numbers.
    \begin{lemma}\label{locdenprodbound}
        Let $m$ be an odd integer such that $m \not\equiv 4 \pmod 3, m \not\equiv 4 \pmod 5$, $p$ be an odd prime, $h = 8(m-2)n + \sum_{j=1}^4\alpha_j(m-4)^2.$
        Then for $p\nmid (m-2)(m-4)$ and $p\nmid \prod_{j=1}^4\alpha_j$ we have the following:
        \begin{equation*}
            \beta_{X^{(p,1,1,1),p}}(h)\leq \frac{2}{p}
        \end{equation*}
        \begin{equation*}
            \beta_{X^{(p,p,1,1),p}}(h)\leq \frac{4}{p^2}
        \end{equation*}
        \begin{equation*}
            \beta_{X^{(p,p,p,1),p}}(h)\leq \frac{8}{p^3}
        \end{equation*}
        \begin{equation*}
            \beta_{X^{(p,p,p,p),p}}(h)\leq 
            \begin{cases}
			\frac{1}{p^2(p-1)}, & \text{for $n \in p\mathbb{Z}_p$ }\\
            \frac{16}{p^4}, & \text{for $n \not\in p\mathbb{Z}_p$ }
		 \end{cases}
        \end{equation*}
        Then we have 
        \begin{equation*}
            \prod_{p\mid \textbf{d}}\beta_{X^{\textbf{d}},p}(h)\leq \frac{\tilde{w}(d_1)\tilde{w}(d_2)\tilde{w}(d_3)\tilde{w}(d_4)}{d_1d_2d_3d_4}
        \end{equation*}
        where $\frac{\tilde{w}(p)}{p}:=\begin{cases}
            \frac{2}{p}, & \text{if $p\nmid n$,}
            \\ \max\{(\frac{1}{p^2(p-1)})^{\frac{1}{4}},\frac{2}{p}\}, & \text{if $p\mid n$}
        \end{cases}$
        \begin{proof}
            Let $$w_p(h):=\frac{1}{1-\psi(p)p^{-2}}\gamma_p(2)^{-1}$$
            Then by \cite{Shimura} we have 
            $$\gamma_p(2)\geq 1-\frac{1}{p}.$$
            Thus $$w_p(h)\leq \frac{1}{1-p^{-2}}\frac{p}{p-1}=\frac{p^3}{(p-1)^2(p+1)} \text{ for } p\geq 5$$
            and $$\leq \frac{p^3}{(p^2+1)(p-1)} \text{ for } p=3$$
            Therefore we have 
            \begin{align*}
              \beta_{X^{(p,1,1,1),p}}(h) &=\frac{1}{p}w_p(h)b_p(h,\lambda_{(p,1,1,1)},0)  
              \\ &\leq \frac{2}{p}.
            \end{align*}
            Other cases can be bound similarly.
        \end{proof}    
    \end{lemma}
    \begin{lemma}
        Let $g(\textbf{d}):=\prod_{p\mid \textbf{d}}\frac{\beta_{X^{\textbf{d}},p}(h)}{\prod_{j=1}^4\beta_{X^{(d_j,1,1,1)},p}(h)}, u_{i,j}:=\gcd(d_i,d_j).$
        Then we have $$g(\textbf{d})\ll \max_{i,j}(u_{i,j})^{12}.$$
        \begin{proof}
            By Lemma \ref{locdenprodbound}, we have $$
\beta_{X^{\textbf{d}},p}(h)\leq \begin{cases}
			(\frac{2}{p})^{\sum_{i=1}^4{\text{ord}_p(d_i)}}, & \text{if $p\nmid n$ }\\
            (\frac{1}{\sqrt{p}})^{\sum_{i=1}^4{\text{ord}_p(d_i)}}, & \text{if $p\mid n$}
		 \end{cases}
$$
and we have 
$$\beta_{X^{(p,1,1,1)},p}(h)\geq \frac{p-1}{2p^2-p+1}.$$
Hence we have
\begin{align*}
    g(\textbf{d})& = \prod_{p\mid u_{1,2}\cdot\cdot\cdot u_{3,4}}\frac{\beta_{X^{\textbf{d}},p}(h)}{\prod_{j=1}^4\beta_{X^{(p,1,1,1)},p}(h)}
    \\&\leq  \prod_{p\mid u_{1,2}\cdot\cdot\cdot u_{3,4}}(\frac{2p^2-p+1}{\sqrt{p}(p-1)})^{\text{ord}_p(d_1)+\text{ord}_p(d_2)+\text{ord}_p(d_3)+\text{ord}_p(d_4)}
    \\& \leq  \prod_{p\mid u_{1,2}\cdot\cdot\cdot u_{3,4}}(\frac{2p^2-p+1}{\sqrt{p}(p-1)})^{4}
    \\&\leq  \prod_{p\mid u_{1,2}\cdot\cdot\cdot u_{3,4}}(\frac{4^4(p-1)^8}{p^2(p-1)^4})
    \\&\leq  \prod_{p\mid u_{1,2}\cdot\cdot\cdot u_{3,4}}4^4p^2 
    \leq 4^4\prod_{i<j \atop 1\leq i,j \leq 4}u_{i,j}^2 \ll \max_{i,j}(u_{i,j})^{12}
\end{align*}
\end{proof}
    \end{lemma}

    \begin{lemma}
        Let $m$ be an odd integer such that $m \not\equiv 4 \pmod 3, m \not\equiv 4 \pmod 5$, $p\geq5$ be an odd prime, $h = 8(m-2)n + \sum_{j=1}^4\alpha_j(m-4)^2.$
        Then for $p\nmid (m-2)(m-4)$ and $p\mid \prod_{j=1}^4\alpha_j$ we have the following:
        \begin{equation*}
            \max\{\beta_{X^{(p,1,1,1)},p}(h),\beta_{X^{(1,p,1,1)},p}(h),\beta_{X^{(1,1,p,1)},p}(h),\beta_{X^{(1,1,1,p)},p}(h)\}\leq \frac{4}{p}
        \end{equation*}
        \begin{multline*}
            \max\{\beta_{X^{(p,p,1,1)},p}(h),\beta_{X^{(p,1,p,1)},p}(h),\beta_{X^{(p,1,1,p)},p}(h),\beta_{X^{(1,p,p,1)},p}(h)
            \\,\beta_{X^{(1,p,1,p)},p}(h),\beta_{X^{(1,1,p,p)},p}(h) \}\leq \frac{2p}{(p-1)^2(p+1)} \leq \frac{16}{p^2}
        \end{multline*}
        \begin{multline*}
            \max\{\beta_{X^{(p,p,p,1)},p}(h),\beta_{X^{(p,p,1,p)},p}(h),\beta_{X^{(p,1,p,p)},p}(h)
            \\,\beta_{X^{(1,p,p,p)},p}(h)\}\leq \frac{2}{(p-1)^2(p+1)} \leq \frac{64}{p^3}
        \end{multline*}
        \begin{equation*}
            \beta_{X^{(p,p,p,p)},p}(h)\leq
            \begin{cases}
			\frac{1}{(p-1)^2(p+1)}, & \text{for $n \in p\mathbb{Z}_p$ }\\
            \frac{256}{p^4}, & \text{for $n \not\in p\mathbb{Z}_p$ }
		 \end{cases}
        \end{equation*}
        \begin{proof}
            By Theorem \ref{Thm:localdenexplict}, the maximum value of $b_p(h,\lambda_{p^{\textbf{c}}},0)$ where $\textbf{c}\in \{(0,0,0,1),(0,0,1,0),(0,1,0,0),(1,0,0,0)\}$ is achieved when $\textbf{t}_d=1, t_1=t_2=0,t_3=1,$ where $N_p=\{1,2,3\},$ and $\textbf{t}_n\geq \textbf{t}_d.$ Then 
            $$b_p(h,\lambda_{p^{\textbf{c}}},0)=1+(1-\frac{1}{p})\delta_p(1)\leq 2-\frac{1}{p}.$$
            Therefore we have 
            \begin{multline*}
                \max\{\beta_{X^{(p,1,1,1)},p}(h),\beta_{X^{(1,p,1,1)},p}(h),\beta_{X^{(1,1,p,1)},p}(h),\beta_{X^{(1,1,1,p)},p}(h)\}
                \\\leq \frac{1}{p}\cdot(2-\frac{1}{p})\cdot\frac{p^3}{(p-1)^2(p+1)}\leq \frac{4}{p} \text{ for } p\geq 5.
            \end{multline*}
            For $\textbf{c}\in \{(0,0,1,1),(0,1,0,1),(1,0,0,1),(0,1,1,0),(1,0,1,0),(1,1,0,0)\},$
            the maximum value of $b_p(h,\lambda_{p^{\textbf{c}}},0)$
            is achieved when $\textbf{t}_d=1, t_1=0,t_2=1,$ where $N_p=\{1,2\},$ $t_3=t_4=1,$ where $D_p=\{3,4\},$ and $\textbf{t}_n< \textbf{t}_d.$
            Then 
            $$b_p(h,\lambda_{p^{\textbf{c}}},0)=1+(\frac{u_1}{p})(\frac{\textbf{u}_n}{p})\leq 2.$$
            Therefore we have 
            \begin{multline*}
            \max\{\beta_{X^{(p,p,1,1)},p}(h),\beta_{X^{(p,1,p,1)},p}(h),\beta_{X^{(p,1,1,p)},p}(h),
            \\\beta_{X^{(1,p,p,1)},p}(h)
            ,\beta_{X^{(1,p,1,p)},p}(h),\beta_{X^{(1,1,p,p)},p}(h) \}
            \\\leq \frac{1}{p^2}\cdot2\cdot\frac{p^3}{(p-1)^2(p+1)}\leq \frac{16}{p^2}.
        \end{multline*}
        For $\textbf{c}\in \{(0,1,1,1),(1,0,1,1),(1,1,0,1),(1,1,1,0)\},$
        the maximum value of $b_p(h,\lambda_{p^{\textbf{c}}},0)$
            is achieved when $\textbf{t}_d=1, t_1=0$ where $N_p=\{1\},$ and $\textbf{t}_n< \textbf{t}_d.$
             Then 
            $$b_p(h,\lambda_{p^{\textbf{c}}},0)=1+(\frac{u_1}{p})(\frac{\textbf{u}_n}{p})\leq 2.$$
            Therefore we have 
            \begin{multline*}
            \max\{\beta_{X^{(p,p,p,1)},p}(h),\beta_{X^{(p,p,1,p)},p}(h),\beta_{X^{(p,1,p,p)},p}(h)
            \\,\beta_{X^{(1,p,p,p)},p}(h)\}\leq \frac{1}{p^3}\cdot2\cdot\frac{p^3}{(p-1)^2(p+1)}\leq \frac{64}{p^3}
            \end{multline*}
            And finally for $\textbf{c}=(1,1,1,1),$ we have $$b_p(h,\lambda_{p^{\textbf{c}}},0)=
            \begin{cases}
                p, & \text{for $n\in p\mathbb{Z}_p$}\\
                0, & \text{for $n\not\in p\mathbb{Z}_p$}
            \end{cases}
            $$
            Therefore we have 
            \begin{equation*}
            \beta_{X^{(p,p,p,p)},p}(h)\leq
            \begin{cases}
			\frac{1}{(p-1)^2(p+1)}, & \text{for $n \in p\mathbb{Z}_p$ }\\
            \frac{256}{p^4}, & \text{for $n \not\in p\mathbb{Z}_p$ }
		 \end{cases}
        \end{equation*}
        \end{proof}
    \end{lemma}

\section{Error from Cusp form}
    We will use the formula from \cite[Lemma 3.2]{Kamaraj}. We provide the statement for the case $l=4.$
    \begin{lemma}\label{cuspexpl}
    Let $M=2(m-2), N_{\boldsymbol{\alpha}}$ be the level of $f_{\boldsymbol{\alpha}}(\textbf{x}).$ Then for any $\delta,\epsilon>0,$ there exists constants $c_{\delta},C_{\epsilon}>0$ such that 
        \begin{multline*}
            \mid a_{G_{X}}(n)\mid\leq \frac{54}{\pi^2\delta_{M^2N_{\boldsymbol{\alpha}}}^\frac{3}{2}}\frac{M^2N_{\boldsymbol{\alpha}}^{2+2\delta}\sqrt{\frac{2\pi}{3}}e^{2\pi}\zeta(1+4\delta)^{\frac{1}{2}}c_{\delta}^{\frac{5}{2}}\varphi(M)}{\prod_{p\mid M,p\nmid N_{\boldsymbol{\alpha}}}(1-p^{-2})^{\frac{1}{2}}\prod_{p\mid N_{\boldsymbol{\alpha}}}(1-p^{-1})^{\frac{1}{2}}}C_{\epsilon}n^{\frac{1}{2}+\epsilon}
            \\ \times(\sum_{d\mid M^2N_{\boldsymbol{\alpha}}}\varphi(\frac{M^2N_{\boldsymbol{\alpha}}}{d})\varphi (d)\frac{M^2N_{\boldsymbol{\alpha}}}{d}(\frac{\gcd(M^2,d)}{M^2})^4)^{\frac{1}{2}}(\frac{27}{\Delta_{\boldsymbol{\alpha}}}\frac{M^2N_{\boldsymbol{\alpha}}}{\pi\delta_{M^2N_{\boldsymbol{\alpha}}}}+16)^{\frac{1}{2}}.
        \end{multline*}
    \end{lemma}
    We next bound the cuspidal contribution.
    \begin{lemma}
    Using Lemma \ref{cuspexpl}, we have
        $$\mid a_{G_{X}}(n)\mid\ll M^{\frac{11}{2}+\epsilon}N_{\boldsymbol{\alpha}}^{\frac{5}{2}+\epsilon}n^{\frac{1}{2}+\epsilon}.$$
        \begin{proof}
            Note that we have 
            \begin{align*}
            \sum_{d\mid M^2N_{\boldsymbol{\alpha}}}\varphi(\frac{M^2N_{\boldsymbol{\alpha}}}{d})\varphi (d)\frac{M^2N_{\boldsymbol{\alpha}}}{d}(\frac{\gcd(M^2,d)}{M^2})^4 &\leq \sum_{d\mid M^2N_{\boldsymbol{\alpha}}}\varphi(\frac{M^2N_{\boldsymbol{\alpha}}}{d})\varphi (d)\frac{M^2N_{\boldsymbol{\alpha}}}{d}
            \\& \leq M^2N_{\boldsymbol{\alpha}}\varphi(M^2N_{\boldsymbol{\alpha}})\sigma_{-1}(M^2N_{\boldsymbol{\alpha}})\ll_\epsilon M^{4+\epsilon}
            \end{align*}
            Therefore we have
            \begin{align*}
                \mid a_{G_{X}}(n)\mid &\ll M^{\frac{5}{2}}N_{\boldsymbol{\alpha}}^{2+\epsilon}\log(\log(N_{\boldsymbol{\alpha}}))\log(\log(M))\cdot M^{2+\epsilon}\cdot MN_{\tilde{a}}^{\frac{1}{2}}n^{\frac{1}{2}+\epsilon}
                \\& \ll M^{\frac{11}{2}+\epsilon}N_{\boldsymbol{\alpha}}^{\frac{5}{2}+\epsilon}n^{\frac{1}{2}+\epsilon}
            \end{align*}
        \end{proof}
    \end{lemma}

\section{Main term and error term from sieving}
    Define $A_{\boldsymbol{\ell}}$ be the set of solutions ${\textbf{x}}\in\mathbb{Z}^4$ with $\ell_j\mid x_j$ to 
    $$\sum_{j=1}^4\alpha_jp_m(x_j)=n.$$
    Let $P$ be a finite set of odd primes and for $\boldsymbol{\ell}\in\mathbb{N}^3$ and $n\in\mathbb{N},$ let $\mathcal{S}_h({A_{\boldsymbol{\ell}}},P)$ denote the size of the set
    \begin{equation}
        \mathcal{F}_h({A_{\boldsymbol{\ell}}},P)=\#\{\textbf{x}\in\mathbb{Z}^4:\sum_{j=1}^4\alpha_jp_m(\ell_jx_j)=n, p\mid x_j\implies p\not\in P\}.  
    \end{equation}
    We then define 
    \begin{equation}
        \mathcal{S}_{{\textbf{c}},h}({A_{\boldsymbol{\ell}}},P):= \sum_{p\mid {2\prod_{j=1}^4\alpha_j}}\sum_{\textbf{b}(p)\in\{0,1\}^4}\prod_{i=1}^4\mu(p^{b(p)_i})\mathcal{S}_h(A_{{\textbf{t}}\cdot\boldsymbol{\ell}},P),
    \end{equation}
    where ${\textbf{t}}=\prod_{p\mid2\prod_{j=1}^4\alpha_j}p^{c_p\cdot{\textbf{b}(p)}}.$
    \\Which is the size of the set 
    \begin{multline*}
    \mathcal{F}_{c,h}({A_{\boldsymbol{\ell}}},P)=\{\textbf{x}\in\mathbb{Z}^4:\sum_{j=1}^4\alpha_jp_m(\ell_jx_j)=n, 
    \\p\mid x_j\implies p\not\in P\cup\{p: p\mid 2\prod_{j=1}^4\alpha_j\} \\\mbox{ or } (p\mid 2\prod_{j=1}^4\alpha_j \mbox{ and } \text{ord}_p(x_j)< c_p)\}.       
    \end{multline*}
    Define $P=P_{z,\Delta_{\alpha}}:=\{p\leq z\}\cap \{p\nmid 2\prod_{j=1}^4\alpha_j\}.$
    \\For $z_0\geq 3,$ define $P(z_0):=\prod_{p<z_0\atop p\nmid 2\prod\alpha_j}p$
    \\ 
    For $\beta, D>0$ and the integer $d$ of the form $d=p_1\cdot\cdot\cdot p_r$ with $p_1 > \cdot\cdot\cdot > p_r$ with $p_j$ an odd prime.
        The Rosser weights $\lambda_d^{\pm}$ are defined as follows:
        \\Let $$y_m=y_m(D,\beta):=(\frac{D}{p_1\cdot\cdot\cdot p_m})^{\frac{1}{\beta}},$$
        then
        \\$\lambda_d^+=\lambda_{d,D}^+(\beta):=
            \begin{cases}
                (-1)^r \quad \text{if $p_{2l+1}<y_{2l+1}(D,\beta) \quad\forall0\leq l\leq \frac{r-1}{2}$},
                \\0 \quad\qquad\text{otherwise.}
            \end{cases}$
            \\$\lambda_d^-=\lambda_{d,D}^-(\beta):=
                \begin{cases}
                (-1)^r \quad \text{if $p_{2l}<y_{2l}(D,\beta) \quad\forall0\leq l\leq \frac{r}{2}$},
                \\0 \quad\qquad\text{otherwise.}
            \end{cases}$
    \\Furthermore define 
    $$\Lambda_d^-:=4\lambda_d^--3\lambda_d^+.$$
    By \cite{Brudern}, we have the following inequalities
    \begin{equation}
        \sum_{d\mid c}\lambda_d^-\leq\sum_{d\mid c}\mu(d)\leq\sum_{d\mid c}\lambda_d^+
    \end{equation}
    and 
    \begin{equation}
        \prod_{j=1}^4\sum_{d_j\mid c_j}\mu(d_j)
        \geq\sum_{k=1}^4\sum_{d_k\mid c_k}\lambda_{d_k}^-\prod_{1\leq 4 \atop{j\neq k}}\sum_{d_j\mid c_j}\lambda_{d_j}^+-
        3\prod_{j=1}^4\sum_{d_j\mid c_j}\lambda_{d_j}^+.
    \end{equation}
    Define $W_{\textbf{c},h}(z_0):=\prod_{p\mid 2\prod_{j=1}^4\alpha_j}C_p\prod_{p\in P_{z_0,\Delta_{\alpha}}}(1-\beta_{X^{(p,1,1,1)},p}(h))$ and 
    $$H(n)=\prod_{p\mid n}(1+p^{-\frac{1}{2}}).$$
    where $C_p$ are the ratio of local densities at $p.$
    For $\epsilon\in \{\pm1\},$ we set 
    $$M_h^\epsilon(z_0)=
        \begin{cases}
        \sum_{d_1\mid P_{z_0,\Delta_\alpha}}\Lambda_{d_1}^{-}\sum_{d_2\mid P_{z_0,\Delta_\alpha}}\lambda_{d_2}^{+}\sum_{d_3\mid P_{z_0,\Delta_\alpha}}\lambda_{d_3}^{+}
        \\\qquad\qquad\qquad\qquad\qquad\cdot\sum_{d_4\mid P_{z_0,\Delta_\alpha}}\lambda_{d_4}^{+}\prod_{p\mid \textbf{d}}\beta_{X^{\textbf{d}},p}(h) & \text{if $\epsilon=-1$},    
        \\\sum_{d_1\mid P_{z_0,\Delta_\alpha}}\lambda_{d_1}^{+}\sum_{d_2\mid P_{z_0,\Delta_\alpha}}\lambda_{d_2}^{+}\sum_{d_3\mid P_{z_0,\Delta_\alpha}}\lambda_{d_3}^{+}
        \\\qquad\qquad\qquad\qquad\qquad\cdot\sum_{d_4\mid P_{z_0,\Delta_\alpha}}\lambda_{d_4}^{+}\prod_{p\mid \textbf{d}}\beta_{X^{\textbf{d}},p}(h) & \text{if $\epsilon=1$},       
        \end{cases}
    $$
    \begin{proposition}
        Let $D_0$ and $z_0$ be given, and write $s_0:=\frac{\log(D_0)}{\log(z_0)}.$ Suppose that all prime factors of $\boldsymbol{\ell}$ are greater than $z_0$ and divide $\prod_{j=1}^4\alpha_j.$ Then for some $B,C>0$
        \begin{align*}
            S_{\textbf{c},h}(A_{\boldsymbol{\ell}},P_{z_0,\Delta_\alpha}))=(W_{c,h}(z_0)+ & O_{\epsilon}(\Delta^{-\frac{1}{2}}H(n)^5\log(z_0)^8
            \\&+\Delta^{\epsilon-B}\log(z_0)^8+\Delta^Ce^{-s_0})r_{gen^+(X^{\boldsymbol{\ell}})}(h)
            \\&+O(\sum_{{\textbf{d}\in S_{\prod_{p\mid 2\prod_{j=1}^4\alpha_j}p^{c_p}}^4\atop{p\mid d_j\implies p\in P_{\delta,\Delta_{\alpha}}}}}\mid r_{X^{{\textbf{d}}\cdot{\boldsymbol{\ell}}}}(h)-r_{gen^+(X^{{\textbf{d}}\cdot{\boldsymbol{\ell}}})}(h)\mid).           
        \end{align*}
    \end{proposition}
    Before proving Proposition 5.1, we require certain inequalities for the upper bound and the lower bound of $S_{\textbf{c},h}(A_{\ell},P_{z_0,\Delta_\alpha}).$
    \begin{proposition}
        Suppose that all prime divisors of $\boldsymbol{\ell}$ are greater than $z_0$ and do not divide $\prod_{j=1}^4\alpha_j.$
        \\(1) We have
            \begin{multline*}
                S_{\textbf{c},h}(A_{\ell},P_{z_0,\Delta_\alpha})\geq
                \prod_{p\mid 2\prod_{j=1}^4\alpha_j}C_pM_h^-(n^{\delta})r_{gen^+(X^{\boldsymbol{\ell}})}(h)
                \\-7\sum_{{\textbf{d}\in S_{\prod_{p\mid 2\prod_{j=1}^4\alpha_j}p^{c_p}}^4\atop{p\mid d_j\implies p\in P_{\delta,\Delta_{\alpha}}}}}\mid (r_{X^{\textbf{d}\cdot\boldsymbol{\ell}}}
                (h)-r_{gen^+(X^{\textbf{d}\cdot\boldsymbol{\ell}})}(h))\mid .
                %\atop{d_j\leq D_0}
            \end{multline*}
        \\(2) We have 
            \begin{multline*}
                S_{\textbf{c},h}(A_{\ell},P_{z_0,\Delta_\alpha})\leq
                \prod_{p\mid 2\prod_{j=1}^4\alpha_j}C_pM_h^+(n^{\delta})r_{gen^+(X^{\boldsymbol{\ell}})}(h)
                \\+7\sum_{{\textbf{d}\in S_{\prod_{p\mid 2\prod_{j=1}^4\alpha_j}p^{c_p}}^4\atop{p\mid d_j\implies p\in P_{\delta,\Delta_{\alpha}}}}}\mid (r_{X^{\textbf{d}\cdot\boldsymbol{\ell}}}
                (h)-r_{gen^+(X^{\textbf{d}\cdot\boldsymbol{\ell}})}(h))\mid .
                %\atop{d_j\leq D_0}
            \end{multline*}
    \end{proposition}
    \begin{proof}
        (1) Suppose $P_1$ and $P_2$ are two disjoint sets of primes with $p\nmid l_j$ for all $p\in P_2$ and $1\leq j\leq 4.$ Writing $R=\prod_{p\in P_2}p,$ we compute
        \begin{align*}
            \mathcal{S}_{\textbf{c},h}(A_\ell,P_1\cup P_2) &=\sum_{x\in F_{\textbf{c},h}(A_\ell,P_1)}\prod_{j=1}^4\sum_{d_j\mid \gcd(x_j,R)}\mu(d_j)  
            \\& =\sum_{x\in F_{\textbf{c},h}(A_\ell,P_1)}\sum_{{d_1\mid R \atop{d_1\mid x_1}}}\mu(d_1) \sum_{{d_2\mid R \atop{d_2\mid x_2}}}\mu(d_2)\sum_{{d_3\mid R \atop{d_3\mid x_3}}}\mu(d_3)\sum_{{d_4\mid R \atop{d_4\mid x_4}}}\mu(d_4)
            \\& \geq \sum_{x\in F_{\textbf{c},h}(A_\ell,P_1)}\sum_{{d_1\mid R \atop{d_1\mid x_1}}}\Lambda^-_{d_1} \sum_{{d_2\mid R \atop{d_2\mid x_2}}}\lambda^+_{d_2}\sum_{{d_3\mid R \atop{d_3\mid x_3}}}\lambda^+_{d_3}\sum_{{d_4\mid R \atop{d_4\mid x_4}}}\lambda^+_{d_4}
            \\&= \sum_{d_1\mid R}\Lambda^-_{d_1}\sum_{d_2\mid R}\lambda^+_{d_2}\sum_{d_3\mid R}\lambda^+_{d_3}\sum_{d_4\mid R}\lambda^+_{d_4}\mathcal{S}_{{\textbf{c}},h}(A_{\textbf{d}\cdot\boldsymbol{\ell}},P_1),
        \end{align*}
        using (5.4) in the second last step.
        We assume that $p\mid \ell_j\implies p>z_0$ and $p\mid \prod_{j=1}^4\alpha_j$ and take $P_1=\emptyset, P_2=P_{z_0,\Delta_\alpha}:=\{p\leq z_0,p\nmid 2\prod_{j=1}^4\alpha_j\},$ then by (5.2) we have
        $$S_{\textbf{c},h}(A_{{\textbf{d}\cdot\boldsymbol{\ell}}},P_1)=\sum_{p\mid {2\prod_{j=1}^4\alpha_j}}\sum_{\textbf{b}(p)\in\{0,1\}^4}\prod_{i=1}^4\mu(p^{b(p)_i})r_{X^{\prod_{p\mid2\alpha_j} p^{c_p\cdot\textbf{b}(p)}\cdot\textbf{d}\cdot\boldsymbol{\ell}}}(h).$$
        We decompose 
        \begin{equation}
            r_{X^{\boldsymbol{d\ell}}}(h)=a_{E_{X^{\boldsymbol{d\ell}}}}(h)+(r_{X^{\boldsymbol{d\ell}}}(h)-a_{E_{X^{\boldsymbol{d\ell}}}}(h)).
        \end{equation}
        Then we have 
        \begin{equation}
            \begin{split}
                S_{\textbf{c},h}(A_{l},P_{\delta,\Delta_\alpha})\geq \sum_{{\textbf{d}\in S_{\prod_{p\mid 2\prod_{j=1}^4\alpha_j}p^{c_p}}^4\atop{p\mid d_j\implies p\in P_{\delta,\Delta_{\alpha}}}}}\Lambda_{d_1}^-\lambda_{d_2}^+\lambda_{d_3}^+\lambda_{d_4}^+(a_{E_{X^{\boldsymbol{d\ell}}}}(h)
                \\+(r_{X^{\boldsymbol{d\ell}}}(h)-a_{E_{X^{\boldsymbol{d\ell}}}}(h))). 
            \end{split}
        \end{equation}
        \\Since $\mid \Lambda_d^-\mid\leq7$ pulling the absolute value inside the second sum in (5.4) yields the last term on the right-hand side of the proposition, and we henceforth ignore it.
        \\Writing 
        $$C_p:=\sum_{\textbf{b}(p)\in\{0,1\}^4}\prod_{i=1}^4\mu(p^{b(p)_i})(\frac{1}{p^{b(p)_1+b(p)_2+b(p)_3+b(p)_4}}\frac{b_p(h,\lambda_{p^{(b(p)_1,b(p)_2,b(p)_3,b(p)_4)}},0)}{b_p(h,\lambda_{\textbf{1}},0)})^{c_p},$$
        we can conclude that 
        \begin{equation}
        \begin{split}
            \sum_{{\textbf{d}\in S_{\prod_{p\mid 2\prod_{j=1}^4\alpha_j}p^{c_p}}^4\atop{p\mid d_j\implies p\in P_{\delta,\Delta_{\alpha}}}}}\Lambda_{d_1}^-\lambda_{d_2}^+\lambda_{d_3}^+\lambda_{d_4}^+a_{E_{X^{\boldsymbol{d\ell}}}}(h)=\prod_{p\mid2\prod\alpha_j}C_p\sum_{{\textbf{d}\in S_{1}^4}\atop{p\mid d_j\implies p\in P_{\delta,\Delta_\alpha}}}\Lambda_{d_1}^-\lambda_{d_2}^+\lambda_{d_3}^+\lambda_{d_4}^+a_{E_{X^{\boldsymbol{d\ell}}}}(h).     
        \end{split}
        \end{equation}
        Thus we may bound the main term from below by
        \begin{equation}
            \prod_{p\mid2\prod\alpha_j}C_p M_h^-(z_0)a_{E_{X^{\boldsymbol{d\ell}}}}(h).
        \end{equation}
        (2) Using the other inequality yields the claim by the same steps as in (1).
    \end{proof}
    In order to prove Proposition 5.1, we need to get bounds for the main term.
    \begin{proof}[Proof of Proposition 5.1]
        The error terms are of the same size in Proposition 5.2 (1) and (2), so
        it remains to bound the main terms.
        We first consider the main term from Proposition 5.2 (2). We separate the terms off which have $\gcd(d_j,d_k) > \Delta$ for some $j \neq k.$
        \begin{align*}
            &\sum_{d_1\mid P(z_0)}\lambda_{d_1}^+\sum_{d_2\mid P(z_0)}\lambda_{d_2}^+\sum_{d_3\mid P(z_0)}\lambda_{d_3}^+\sum_{d_4\mid P(z_0)\atop \gcd(d_3,d_4)>\Delta}\lambda_{d_4}^+\prod_{p\mid \textbf{d}}\beta_{X^{\textbf{d}},p}(h)
            \\& \leq \sum_{d_1\mid P(z_0)}\sum_{d_2\mid P(z_0)}\sum_{d_3\mid P(z_0)}\sum_{d_4\mid P(z_0)\atop \gcd(d_3,d_4)>\Delta}\prod_{p\mid \textbf{d}}\beta_{X^{\textbf{d}},p}(h)
            \\& \ll(\sum_{d_1\mid P(z_0)}\frac{\tilde{w}(d_1)}{d_1})^2\sum_{\epsilon\geq\Delta}(\sum_{d_2\mid P(z_0)\atop \epsilon\mid d_2}\frac{\tilde{w}(d_2)}{d_2})^2
            \\& \ll \prod_{p\in P_{z_0,\Delta_{\alpha}}}(1+\frac{\tilde{w}(p)}{p})^4\sum_{\epsilon\geq\Delta}\mu(\epsilon)^2(\frac{\tilde{w}(\epsilon)}{\epsilon})^2
        \end{align*}
        By Lemma 3.7 
        \begin{align*}     \sum_{\epsilon\geq\Delta}\mu(\epsilon)^2(\frac{\tilde{w}(\epsilon)}{\epsilon})^2 &\leq \sum_{\epsilon\geq\Delta}(\frac{\epsilon}{\Delta})^{\frac{1}{2}}\mu(\epsilon)^2(\frac{\tilde{w}(\epsilon)}{\epsilon})^2
            \leq \Delta^{-\frac{1}{2}}\sum_{\epsilon\geq 1}\mu(\epsilon)^2(\frac{\tilde{w}(\epsilon)^2}{\epsilon^{\frac{3}{2}}})
            \\& =\Delta^{-\frac{1}{2}}\prod_{p}(1+\frac{\tilde{w}(p)^2}{p^{\frac{3}{2}}})\leq \Delta^{-\frac{1}{2}}\prod_{p\mid n}(1+\frac{\tilde{w}(p)^2}{p^{\frac{3}{2}}})\prod_{p}(1+\frac{4}{p^{\frac{3}{2}}})
            \\& \ll\Delta^{-\frac{1}{2}}\prod_{p\mid n}(1+\frac{\tilde{w}(p)^2}{p^{\frac{3}{2}}}) \ll\Delta^{-\frac{1}{2}}H(n).
        \end{align*}
        Then we have 
        $$\prod_{p\in P_{z_0,\Delta_{\alpha}}}(1+\frac{\tilde{w}(p)}{p})\leq \prod_{p\leq z_0}(1+\frac{\tilde{w}(p)}{p}) \ll H(n)\log(z_0)^2.$$
        Hence we may write $M_h^+(z_0)$ as 
        $$\sum_{d_1\mid P(z_0)}\lambda_{d_1}^+\sum_{d_2\mid P(z_0)}\lambda_{d_2}^+\sum_{d_3\mid P(z_0)}\lambda_{d_3}^+\sum_{d_4\mid P(z_0)\atop \gcd(d_3,d_4)\leq \Delta}\lambda_{d_4}^+\prod_{p\mid \textbf{d}}\beta_{X^{\textbf{d}},p}(h)+O(\Delta^{-\frac{1}{2}}H(n)^5\log(z_0)^8)$$
        By grouping those $\textbf{d}$ with $\gcd(d_i,d_j)=u_{i,j}.$ We have
        \begin{align*}
            \sum_{d_1\mid P(z_0)}\lambda_{d_1}^+ &\sum_{d_2\mid P(z_0)}\lambda_{d_2}^+\sum_{d_3\mid P(z_0)}\lambda_{d_3}^+\sum_{d_4\mid P(z_0)\atop \gcd(d_3,d_4)\leq \Delta}\lambda_{d_4}^+\prod_{p\mid \textbf{d}}\beta_{X^{\textbf{d}},p}(h)
            \\& =\sum_{u_{1,2}\leq \Delta}\cdot\cdot\cdot\sum_{u_{3,4}\leq \Delta}\sum_{d_1\mid P(z_0)}\lambda_{d_1}^+\sum_{d_2\mid P(z_0)}\lambda_{d_2}^+\sum_{d_3\mid P(z_0)}\lambda_{d_3}^+\sum_{d_4\mid P(z_0)\atop \gcd(d_3,d_4)=u_{i,j}}\lambda_{d_4}^+\prod_{p\mid \textbf{d}}\beta_{X^{\textbf{d}},p}(h).
        \end{align*}
        Since $g(\textbf{d})$ only depends on $u_{i,j}$ by Lemma 3.8, we may write 
        \begin{multline*}
            \sum_{u_{1,2}\leq \Delta} \cdot\cdot\cdot\sum_{u_{3,4}\leq \Delta}\sum_{d_1\mid P(z_0)}\lambda_{d_1}^+\sum_{d_2\mid P(z_0)}\lambda_{d_2}^+\sum_{d_3\mid P(z_0)}\lambda_{d_3}^+\sum_{d_4\mid P(z_0)\atop \gcd(d_3,d_4)=u_{i,j}}\lambda_{d_4}^+\prod_{p\mid \textbf{d}}\beta_{X^{\textbf{d}},p}(h)
            \\ = \sum_{u_{1,2}\leq \Delta}\cdot\cdot\cdot\sum_{u_{3,4}\leq \Delta}g(\textbf{d})(\sum_{d_1\mid P(z_0)}\lambda_{d_1}^+\prod_{p\mid \textbf{d}_1}\beta_{X^{(p,1,1,1)},p}(h)\sum_{d_2\mid P(z_0)}\lambda_{d_2}^+\prod_{p\mid \textbf{d}_2}\beta_{X^{(1,p,1,1)},p}(h)
            \\\cdot\sum_{d_3\mid P(z_0)}\lambda_{d_3}^+\prod_{p\mid \textbf{d}_3}\beta_{X^{(1,1,p,1)},p}(h)\sum_{d_4\mid P(z_0)\atop \gcd(d_3,d_4)=u_{i,j}}\lambda_{d_4}^+\prod_{p\mid \textbf{d}_4}\beta_{X^{(1,1,1,p)},p}(h)).
        \end{multline*}
        We rewrite the $\gcd$ condition on the inner sums by excluding divisibility by other primes using $\mu$ (or inclusion/exclusion). Letting $l_{i,j}$ run through all possible divisors of $P(z_0)$ relatively prime to $u_{i,j}$, this gives
        \begin{multline*}
            \sum_{d_1\mid P(z_0)}\lambda_{d_1}^+\prod_{p\mid \textbf{d}_1}\beta_{X^{(p,1,1,1)},p}(h)\sum_{d_2\mid P(z_0)}\lambda_{d_2}^+\prod_{p\mid \textbf{d}_2}\beta_{X^{(1,p,1,1)},p}(h)
            \\\cdot\sum_{d_3\mid P(z_0)}\lambda_{d_3}^+\prod_{p\mid \textbf{d}_3}\beta_{X^{(1,1,p,1)},p}(h)\sum_{d_4\mid P(z_0)\atop \gcd(d_3,d_4)=u_{i,j}}\lambda_{d_4}^+\prod_{p\mid \textbf{d}_4}\beta_{X^{(1,1,1,p)},p}(h)
        \end{multline*}
        \begin{multline*}
            =\sum_{l_{i,j}\atop 1\leq i,j\leq4}\mu(l_{i,j})\sum_{d_1\mid P(z_0)\atop lcm(u_{1,2}l_{1,2},u_{1,3}l_{1,3},u_{1,4}l_{1,4})}\lambda_{d_1}^+\prod_{p\mid \textbf{d}_1}\beta_{X^{(p,1,1,1)},p}(h)
            \\\cdot\cdot\cdot\sum_{d_4\mid P(z_0)\atop lcm(u_{1,4}l_{1,4},u_{2,4}l_{2,4},u_{3,4}l_{3,4})}\lambda_{d_4}^+\prod_{p\mid \textbf{d}_4}\beta_{X^{(1,1,1,p)},p}(h).
        \end{multline*}
        By Lemma 3.7 we have
        \begin{multline*}
            \sum_{d_1\mid P(z_0)\atop \delta\mid d_1}\lambda_{d_1}^+\prod_{p\mid \textbf{d}_1}\beta_{X^{(p,1,1,1)},p}(h) \leq \prod_{p\mid \delta}\beta_{X^{(\delta,1,1,1)},p}(h)\cdot \prod_{p\mid P(z_0)}(1+\frac{2}{p})
            \\ \ll \frac{2^{w(\delta)}}{\varphi(\delta)}\log(z_0)^2
            \ll \frac{\sigma(\delta)^2}{\delta}\log(z_0)^2.
        \end{multline*}
        Then we want to bound 
        $$\sum_{l_{1,2},\cdot\cdot\cdot,l_{3,4}\leq D_0\atop \Delta^{B}<l_{1,2}}
        \frac{\mu(l_{1,2})^2\cdot\cdot\cdot \mu(l_{3,4})^2\sigma_0(l_{1,2})^4\cdot\cdot\cdot \sigma_0(l_{3,4})^4\gcd(l_{1,2},l_{1,3})\cdot\cdot\cdot\gcd(l_{2,4},l_{3,4})}{l_{1,2}^2\cdot\cdot\cdot l_{3,4}^2\gcd(l_{1,2},l_{1,3},l_{1,4})\cdot\cdot\cdot\gcd(l_{1,4},l_{2,4},l_{3,4})}$$
        We then bound
        \begin{multline*}
            \sum_{l_{3,4}\leq D_0}
            \frac{ \mu(l_{3,4})^2\sigma_0(l_{3,4})^4\gcd(l_{1,3},l_{3,4})\gcd(l_{1,4},l_{3,4})\gcd(l_{2,4},l_{3,4})}{l_{3,4}^2\gcd(l_{1,3},l_{2,3},l_{3,4})\gcd(l_{1,4},l_{2,4},l_{3,4})}
            \\ \leq \prod_{p\mid l_{1,3}l_{2,3}l_{1,4}l_{2,4}}(2^4+1)\prod_{p}(1+\frac{2^4}{p^2}) \ll (\sigma_0(l_{1,3})\sigma_0(l_{1,4})\sigma_0(l_{2,3})\sigma_0(l_{2,4}))^{5}
        \end{multline*}
        Then bound 
        \begin{multline*}
            \sum_{l_{2,4}\leq D_0}
            \frac{ \mu(l_{2,4})^2\sigma_0(l_{2,4})^{9}\gcd(l_{1,2},l_{2,4})\gcd(l_{2,3},l_{2,4})\gcd(l_{1,4},l_{2,4})}{l_{2,4}^2\gcd(l_{1,2},l_{2,3},l_{2,4})}
            \\ \leq \prod_{p\mid l_{1,3}l_{2,3}l_{1,4}}(2^{9}+1)\prod_{p}(1+\frac{2^{9}}{p^2}) \ll (\sigma_0(l_{1,2})\sigma_0(l_{2,3})\sigma_0(l_{1,4}))^{10}
        \end{multline*}
        Then bound 
        \begin{multline*}
            \sum_{l_{1,4}\leq D_0}
            \frac{ \mu(l_{1,4})^2\sigma_0(l_{1,4})^{19}\gcd(l_{1,2},l_{1,4})\gcd(l_{1,3},l_{1,4})}{l_{1,4}^2\gcd(l_{1,2},l_{1,3},l_{1,4})}
            \\ \leq \prod_{p\mid l_{1,2}l_{1,3}}(2^{19}+1)\prod_{p}(1+\frac{2^{19}}{p^2}) \ll (\sigma_0(l_{1,2})\sigma_0(l_{1,3}))^{20}
        \end{multline*}
        Then bound 
        \begin{multline*}
            \sum_{l_{2,3}\leq D_0}
            \frac{ \mu(l_{2,3})^2\sigma_0(l_{2,3})^{19}\gcd(l_{1,2},l_{2,3})\gcd(l_{1,3},l_{2,3})}{l_{2,3}^2}
            \\ \leq \prod_{p\mid l_{1,2}}(2^{19}+1)\prod_{p}(1+\frac{2^{19}}{p^2}) \ll \sigma_0(l_{1,2})^{20}
        \end{multline*}
        Finally the sum over $i_{1,2}$ is bounded by 
            $$\sum_{l_{1,2} > \Delta^{B}}
            \frac{ \mu(l_{1,2})^2\sigma_0(l_{1,2})^{54}}{l_{1,2}^2}
             \ll_\epsilon \Delta^{\epsilon-B}.$$
        Therefore the contribution from $l_{i,j}$ with at least one of them $> \Delta^{B}$ is 
        $$\ll_\epsilon \Delta^{\epsilon-B} \log(z_0)^8.$$
        And we finally use [2, Lemma 11] to finish the proof for the upper bound. Since we have taken the absolute value at each step, the proof for the lower bound follows by the same argument, using Proposition 5.2 (1) as a starting point.  
        \end{proof}

\section{Sieving and proof of Theorem 1.1}
    We employ a second sieve to prove Theorem 1.1.
    \begin{proof}[Proof of Theorem 1.1] 
    Suppose that $z_0 < z$ and set $P_h(z_0, z) := \prod_{z_0<p\leq z \atop p\mid\prod_{j=1}^4\alpha_j}p.$ Then we have
    \begin{align*}
        S_{c,h}(A_{\textbf{1},z}) &=\sum_{F_{\textbf{c,h}}(A_{\textbf{1}},z_0)}\prod_{j=1}^4\sum_{l_{j}\mid \gcd(x_j,P_h(z_0, z))}\mu(l_j)
        \\&= \sum_{F_{\textbf{c,h}}(A_{\textbf{1}},z_0)}\sum_{l_{1}\mid P_h(z_0, z)\atop l_1\mid x_1}\mu(l_1)\sum_{l_{2}\mid P_h(z_0, z)\atop l_2\mid x_2}\mu(l_2)\sum_{l_{3}\mid P_h(z_0, z)\atop l_3\mid x_3}\mu(l_3)\sum_{l_{4}\mid P_h(z_0, z)\atop l_4\mid x_4}\mu(l_4)
        \\& \geq \sum_{F_{\textbf{c,h}}(A_{\textbf{1}},z_0)}\sum_{l_{1}\mid P_h(z_0, z)\atop l_1\mid x_1}\Lambda_{l_1}^-\sum_{l_{2}\mid P_h(z_0, z)\atop l_2\mid x_2}\lambda_{l_2}^+\sum_{l_{3}\mid P_h(z_0, z)\atop l_3\mid x_3}\lambda_{l_3}^+\sum_{l_{4}\mid P_h(z_0, z)\atop l_4\mid x_4}\lambda_{l_4}^+
        \\& =\sum_{l_{1}\mid P_h(z_0, z)}\Lambda_{l_1}^-\sum_{l_{2}\mid P_h(z_0, z)}\lambda_{l_2}^+\sum_{l_{3}\mid P_h(z_0, z)}\lambda_{l_3}^+\sum_{l_{4}\mid P_h(z_0, z)}\lambda_{l_4}^+S_{\textbf{c},h}(A_{\boldsymbol{\ell}},z_0).
    \end{align*}
    We then plug in Proposition 5.1 in the inner sum to obtain 
    \begin{multline*}
        S_{\textbf{c},h}(A_{\textbf{1}},z)\geq \sum_{\ell_{1}\mid P_h(z_0, z)}\Lambda_{\ell_1}^-\sum_{\ell_{2}\mid P_h(z_0, z)}\lambda_{\ell_2}^+\sum_{\ell_{3}\mid P_h(z_0, z)}\lambda_{\ell_3}^+\sum_{\ell_{4}\mid P_h(z_0, z)}\lambda_{\ell_4}^+ (W_{c,h}(n^\delta)+  
        \\O_{\epsilon}(\Delta^{-\frac{1}{2}}H(n)^5\log(n^\delta)
        +\Delta^{\epsilon-B}\log(n^\delta)^8+\Delta^Ce^{-s_0}))a_{E_{X^{\boldsymbol{d\ell}}}}(h)
        \\+O(\sum_{{\boldsymbol{\ell}\in S_{\prod_{p\mid 2\prod_{j=1}^4\alpha_j}p^{c_p}}^4\atop{p\mid \ell_j\implies p\mid P_{h}(z_0,z) \atop \ell_j\leq D}}}
        \sum_{{\textbf{d}\in S_{\prod_{p\mid 2\prod_{j=1}^4\alpha_j}p^{c_p}}^4\atop{p\mid d_j\implies p\in P_{\delta,\Delta_{\alpha}}}}\atop d_j\leq D_0}\mid r_{X^{{\textbf{d}}\cdot{\boldsymbol{\ell}}}}(h)-a_{E_{X^{\boldsymbol{d\ell}}}}(h)\mid)).
    \end{multline*}
    We first bound the last O-term. Lemma 4.2 implies that
    \begin{align*}
        \sum_{{\boldsymbol{\ell}\in S_{\prod_{p\mid 2\prod_{j=1}^4\alpha_j}p^{c_p}}^4\atop{p\mid l_j\implies p\mid P_{h}(z_0,z) \atop l_j\leq D}}}
        &\sum_{{\textbf{d}\in S_{\prod_{p\mid 2\prod_{j=1}^4\alpha_j}p^{c_p}}^4\atop{p\mid d_j\implies p\in P_{\delta,\Delta_{\alpha}}}}\atop d_j\leq D_0}\mid r_{X^{{\textbf{d}}\cdot{\boldsymbol{\ell}}}}(h)-a_{E_{X^{\boldsymbol{d\ell}}}}(h)\mid
        \\& \ll \sum_{{\boldsymbol{\ell}\in S_{\prod_{p\mid 2\prod_{j=1}^4\alpha_j}p^{c_p}}^4\atop{p\mid l_j\implies p\mid P_{h}(z_0,z) \atop l_j\leq D}}}
        \sum_{{\textbf{d}\in S_{\prod_{p\mid 2\prod_{j=1}^4\alpha_j}p^{c_p}}^4\atop{p\mid d_j\implies p\in P_{\delta,\Delta_{\alpha}}}}\atop d_j\leq D_0}N_{\alpha,\boldsymbol{d\ell}}^{\frac{5}{2}+\epsilon}h^{\frac{1}{2}+\epsilon}
    \end{align*}
    where $N_{\alpha,\boldsymbol{d\ell}}$ is the level of $f_{\boldsymbol{\alpha}}(\boldsymbol{d\ell}\textbf{x}).$
    Since $\ell_j\leq D$, $d_j\leq D_0,$ $N_{\alpha,\boldsymbol{d\ell}}\ll_m\prod_{j=1}^4d_j^2l_j^2\alpha_j,$ we have 
    \begin{align}
        \sum_{{\boldsymbol{\ell}\in S_{\prod_{p\mid 2\prod_{j=1}^4\alpha_j}p^{c_p}}^4\atop{p\mid \ell_j\implies p\mid P_{h}(z_0,z) \atop \ell_j\leq D}}}
        &\sum_{{\textbf{d}\in S_{\prod_{p\mid 2\prod_{j=1}^4\alpha_j}p^{c_p}}^4\atop{p\mid d_j\implies p\in P_{\delta,\Delta_{\alpha}}}}\atop d_j\leq D_0}N_{\alpha,\boldsymbol{d\ell}}^{\frac{11}{4}+\epsilon}h^{\frac{3}{5}}
        \ll_m(DD_0)^4(DD_0)^{22+\epsilon}(\prod_{j=1}^4\alpha_j)^{\frac{11}{4}+\epsilon}h^{\frac{1}{2}+\epsilon}
    \end{align}
    $$\qquad\qquad\qquad\qquad\qquad\qquad\qquad=(DD_0)^{26+\epsilon}(\prod_{j=1}^4\alpha_j)^{\frac{11}{4}+\epsilon}h^{\frac{1}{2}+\epsilon}.$$
    For the error terms from the main term, we again note that they are independent of $\boldsymbol{\ell}$, so, using [1, (7.2)], they may be bounded by
    \begin{equation}
        (\Delta^{-\frac{1}{2}}H(n)^5\log(n^\delta)
        +\Delta^{\epsilon-B}\log(n^\delta)^8+\Delta^Ce^{-s_0})\sum_{{\boldsymbol{\ell}\in S_{\prod_{p\mid 2\prod_{j=1}^4\alpha_j}p^{c_p}}^4\atop{p\mid \ell_j\implies p\mid P_{h}(z_0,z) \atop l_j\leq D}}}a_{E_{X^{\boldsymbol{d\ell}}}}(h)
    \end{equation}
    $$\ll (\Delta^{-\frac{1}{2}}H(n)^5\log(n^\delta)
        +\Delta^{\epsilon-B}\log(n^\delta)^8+\Delta^Ce^{-s_0})(\frac{\log(z)}{\log(z_0)})^{6}a_{E_{X^{\boldsymbol{d\ell}}}}(h).$$
    For the main term, we write 
    \begin{align}
        \sum_{\ell_{1}\mid P_h(z_0, z)}\Lambda_{\ell_1}^-\sum_{\ell_{2}\mid P_h(z_0, z)}\lambda_{\ell_2}^+\sum_{\ell_{3}\mid P_h(z_0, z)}\lambda_{\ell_3}^+\sum_{\ell_{4}\mid P_h(z_0, z)}\lambda_{\ell_4}^+\prod_{p\mid\boldsymbol{\ell}}\beta_{X^{\boldsymbol{\ell}},p}(h)
    \end{align}
    \begin{multline*}
        =\sum_{\ell_{1}\mid P_h(z_0, z)}\Lambda_{\ell_1}^-\sum_{\ell_{2}\mid P_h(z_0, z)}\lambda_{\ell_2}^+\sum_{\ell_{3}\mid P_h(z_0, z)}\lambda_{\ell_3}^+\sum_{\ell_{4}\mid P_h(z_0, z)\atop \gcd(\ell_i,\ell_j)=1}\lambda_{\ell_4}^+\prod_{p\mid\boldsymbol{\ell}}\beta_{X^{\boldsymbol{\ell}},p}(h)
        \\+ \sum_{\boldsymbol{\ell}\mid P_{h}(z_0,z)^4\atop \exists(j,k):\gcd(\ell_j,\ell_k)\neq 1}\Lambda_{\ell_1}^-\lambda_{\ell_2}^+\lambda_{\ell_3}^+\lambda_{\ell_4}^+\prod_{p\mid\boldsymbol{\ell}}\beta_{X^{\boldsymbol{\ell}},p}(h)
    \end{multline*}
    \begin{multline*}
        =\sum_{\ell_1\mid P_{h}(z_0,z)}\Lambda_{\ell_1}^-\prod_{p\mid \ell_1}\beta_{X^{(\ell_1,1,1,1)},p}(h)\sum_{\ell_2\mid P_{h}(z_0,z)}\lambda_{\ell_2}^+\prod_{p\mid \ell_2}\beta_{X^{(1,\ell_1,1,1)},p}(h)
        \\\cdot\sum_{\ell_3\mid P_{h}(z_0,z)}\lambda_{\ell_3}^+\prod_{p\mid \ell_3}\beta_{X^{(1,1,\ell_1,1)},p}(h)\sum_{\ell_4\mid P_{h}(z_0,z)}\lambda_{\ell_4}^+\prod_{p\mid \ell_4}\beta_{X^{(1,1,1,\ell_1)},p}(h)
        \\+\sum_{\boldsymbol{\ell}\mid p_{h}(z_0,z)^4\atop \exists(j,m):\gcd(\ell_j,\ell_k)\neq 1}\Lambda_{\ell_1}^-\lambda_{\ell_2}^+\lambda_{\ell_3}^+\lambda_{\ell_4}^+(\prod_{p\mid\boldsymbol{\ell}}\beta_{X^{\boldsymbol{\ell}},p}(h)-\prod_{p\mid \ell_1}\beta_{X^{(\ell_1,1,1,1)},p}(h)\prod_{p\mid \ell_2}\beta_{X^{(1,\ell_1,1,1)},p}(h)
        \\\cdot\prod_{p\mid \ell_3}\beta_{X^{(1,1,\ell_1,1)},p}(h)\prod_{p\mid \ell_4}\beta_{X^{(1,1,1,\ell_1)},p}(h)).
    \end{multline*}
    Abbreviating $\beta_{\ell_1,g}:=\prod_{p\mid\gcd(\ell_1,g)}\beta_{X^{(\ell_1,1,1,1)},p}(h),...,\beta_{\ell_4,g}:=\prod_{p\mid\gcd(\ell_4,g)}\beta_{X^{(1,1,1,\ell_1)},p}(h)$ Then we can bound the last sum as
    \begin{equation}
        \ll \sum_{g\mid P_h(z_0,z)\atop g\neq 1}\sum_{\substack{\boldsymbol{\ell}\mid P_h(z_0,z)^4\\ g_1=\gcd(\ell_1,\ell_2)\mid g \\...\\g_6=\gcd(\ell_3,\ell_4)\mid g)}}\prod_{p\mid \boldsymbol{\ell}\atop p \nmid g}\beta_{X^{\boldsymbol{\ell}},p}(h)\mid\prod_{p\mid g}\beta_{X^{\boldsymbol{\ell}},p}(h) - \beta_{\ell_1,g}\beta_{\ell_2,g}\beta_{\ell_3,g}\beta_{\ell_4,g}\mid.
    \end{equation}
    We then bound the absolute value by
    \begin{equation}
        \mid\prod_{p\mid g}\beta_{X^{\boldsymbol{\ell}},p}(h) - \beta_{\ell_1,g}\beta_{\ell_2,g}\beta_{\ell_3,g}\beta_{\ell_4,g}\mid
        \leq \max\{\prod_{p\mid g}\beta_{X^{\boldsymbol{\ell}},p}(h), \beta_{\ell_1,g}\beta_{\ell_2,g}\beta_{\ell_3,g}\beta_{\ell_4,g}\}
    \end{equation}
    $$\leq \prod_{p\mid g}\max\{\max\{\beta_{X^{\boldsymbol{\ell}},p}(h)\}, \max_i\{\beta_{\ell_i,g}\}^{\#\{j:p\mid \ell_j\}}\}.$$
    The product after sum becomes
    \begin{equation}
        \leq \prod_{p\mid l \atop \#\{j:p\mid \ell_j\}=1}\max_i\{\beta_{\ell_i,p}(h)\}
        \cdot \prod_{p\mid g \atop \#\{j:p\mid \ell_j\}=2}\max\{\max\{\beta_{X^{(p,p,1,1)},p}(h),...,\beta_{X^{(1,1,p,p)},p}(h)\},\max_i\{\beta_{\ell_i,g}\}^{2}\}
    \end{equation}
    \begin{equation*}
        \cdot \prod_{p\mid g \atop \#\{j:p\mid \ell_j\}=3}\max\{\max\{\beta_{X^{(p,p,p,1)},p}(h),...,\beta_{X^{(1,p,p,p)},p}(h)\},\max_i\{\beta_{\ell_i,g}\}^{3}\}
    \end{equation*}
    \begin{equation*}
        \cdot \prod_{p\mid g \atop \#\{j:p\mid \ell_j\}=4}\max\{\beta_{X^{(p,p,p,p)},p}(h)\},\max_i\{\beta_{\ell_i,g}\}^{4}\}.
    \end{equation*}
    Without loss of generality, taking $z_0$ sufficiently large, we may assume that $p \nmid 2(m-2)(m-4)$ for every $p \mid \boldsymbol{\ell}.$ Thus, using Lemma 3.9, we may bound (6.6) against
    \begin{align*}
        &\leq \prod_{p\mid {\boldsymbol{\ell}}\atop \#\{j:p\mid \ell_j\}=1}\frac{4}{p}\cdot \prod_{p\mid g\atop \#\{j:p\mid \ell_j\}=2}\frac{16}{p^2}\cdot\prod_{p\mid g\atop \#\{j:p\mid \ell_j\}=3}\frac{64}{p^3}\cdot\prod_{\substack{p\mid {\boldsymbol{\ell}}\\ \#\{j:p\mid \ell_j\}=1\\p\nmid n }}\frac{256}{p^4}\cdot \prod_{\substack{p\mid {\boldsymbol{\ell}}\\ \#\{j:p\mid \ell_j\}=1\\p\mid n }}\frac{1}{(p-1)^2(p+1)}
        \\&=\frac{1}{l_1l_2l_3l_4}\prod_{p\mid {\boldsymbol{\ell}}\atop \#\{j:p\mid \ell_j\}=1}4 \prod_{p\mid g\atop \#\{j:p\mid \ell_j\}=2}4\prod_{p\mid g\atop \#\{j:p\mid \ell_j\}=3}16\prod_{\substack{p\mid {\boldsymbol{\ell}}\\ \#\{j:p\mid \ell_j\}=1\\p\nmid n }}64 \prod_{\substack{p\mid {\boldsymbol{\ell}}\\ \#\{j:p\mid \ell_j\}=1\\p\mid n }}\frac{p^4}{4(p-1)^2(p+1)}
        \\&\leq \frac{4^{w^{(\ell_1\ell_2\ell_3\ell_4)}}64^{w(g)}}{\ell_1\ell_2\ell_3\ell_4}\gcd(g,n,\ell_1,\ell_2,\ell_3,\ell_4)
        \\ &\leq\frac{4^{w(\ell_1)}}{l_1}\frac{4^{w(\ell_2)}}{l_2}\frac{4^{w(\ell_3)}}{l_3}\frac{4^{w(\ell_4)}}{l_4}64^{w(g)}\gcd(g,n,\ell_1,\ell_2,\ell_3,\ell_4).
    \end{align*}
     
    Hence we can bound (6.4) by 
    \begin{equation}
        \sum_{g\mid P_{h}(z_0,z)}\sum_{\substack{\boldsymbol{\ell}\mid P_h(z_0,z)^4
        \\ g_{1,2}=\gcd(\ell_1,\ell_2)\mid g\\...\\g_{3,4}=\gcd(\ell_3,\ell_4)\mid g \\ lcm(g_{1,2},...,g_{3,4})=g}}\frac{4^{w(\ell_1)}}{\ell_1}\frac{4^{w(\ell_2)}}{\ell_2}\frac{4^{w(\ell_3)}}{\ell_3}\frac{4^{w(\ell_4)}}{\ell_4}64^{w(g)}\gcd(g,n,\ell_1,\ell_2,\ell_3,\ell_4).
    \end{equation}
    Let $g_1:=\text{lcm}(g_{1,2},g_{1,3},g_{1,4}),...,g_4:=\text{lcm}(g_{1,4},g_{2,4},g_{3,4}),$ and since $g_j\mid \ell_j,$ by a change of variable $\ell_j \xrightarrow{} g_j\ell_j,$ (6.7) become
    \begin{multline*}
        \sum_{\substack{g\mid P_{h}(z_0,z)\\g\neq 1\\g_1,g_2,g_3,g_4\mid g\\ \text{lcm}(\gcd(g_1,g_2),...,\gcd(g_3,g_4))=g}}\frac{4^{w(g_1)}}{g_1}\frac{4^{w(g_2)}}{g_2}\frac{4^{w(g_3)}}{g_3}\frac{4^{w(g_4)}}{g_4}64^{w(g)}\gcd(n,g_1,g_2,g_3,g_4).
        \\\cdot\sum_{\substack{\boldsymbol{\ell}\mid P_h(z_0,z)^4
        \\ g_{1,2}=\gcd(\ell_1,\ell_2)\mid g\\...\\g_{3,4}=\gcd(\ell_3,\ell_4)\mid g \\ lcm(g_{1,2},...,g_{3,4})=g}}\frac{4^{w(\ell_1)}}{\ell_1}\frac{4^{w(\ell_2)}}{\ell_2}\frac{4^{w(\ell_3)}}{\ell_3}\frac{4^{w(\ell_4)}}{\ell_4}
        \\\leq \sum_{\substack{g\mid P_{h}(z_0,z)\\g\neq 1\\g_1,g_2,g_3,g_4\mid g\\ \text{lcm}(\gcd(g_1,g_2),...,\gcd(g_3,g_4))=g}}\prod_{p\mid g}
        \frac{4^{4}64p^{\text{min}\{\text{ord}_p(g_1),\text{ord}_p(g_2),\text{ord}_p(g_3),\text{ord}_p(g_4)\}}}{p^{\text{ord}_p(g_1)+\text{ord}_p(g_2)+\text{ord}_p(g_3)+\text{ord}_p(g_4)}}
        \\\cdot\sum_{\boldsymbol{\ell}\mid P_h(z_0,z)^4
        }\frac{4^{w(\ell_1)}}{\ell_1}\frac{4^{w(\ell_2)}}{\ell_2}\frac{4^{w(\ell_3)}}{\ell_3}\frac{4^{w(\ell_4)}}{\ell_4}
    \end{multline*}
    We then note that for $p\mid g$
    \begin{multline*}
    2\leq \text{ord}_p(g_1)+\text{ord}_p(g_2)+\text{ord}_p(g_3)+\text{ord}_p(g_4)
    \\-\text{min}\{\text{ord}_p(g_1),\text{ord}_p(g_2),\text{ord}_p(g_3),\text{ord}_p(g_4)\}\leq 3       
    \end{multline*}
     because $\text{lcm}(g_{1,2},...,g_{3,4})=g.$ Writing 
     \begin{multline*}
       \sum_{\boldsymbol{\ell}\mid P_h(z_0,z)^4}\frac{4^{w(\ell_1)}}{\ell_1}\frac{4^{w(\ell_2)}}{\ell_2}\frac{4^{w(\ell_3)}}{\ell_3}\frac{4^{w(\ell_4)}}{\ell_4}=\prod_{p\mid P_h(z_0,z)}(1+\frac{4}{p})^4  
       \\\leq \prod_{p\mid P_h(z_0,z)}(1+\frac{1}{p})^{16} \ll (\frac{\log(z)}{\log(z_0)})^{16}.
     \end{multline*}
     And writing 
     \begin{multline*}
         \sum_{\substack{g\mid P_{h}(z_0,z)\\g\neq 1\\g_1,g_2,g_3,g_4\mid g\\ \text{lcm}(\gcd(g_1,g_2),...,\gcd(g_3,g_4))=g}}\prod_{p\mid g}
        \frac{4^364p^{\text{min}\{\text{ord}_p(g_1),\text{ord}_p(g_2),\text{ord}_p(g_3),\text{ord}_p(g_4)\}}}{p^{\text{ord}_p(g_1)+\text{ord}_p(g_2)+\text{ord}_p(g_3)+\text{ord}_p(g_4)}}
        \\\leq \sum_{\substack{g\mid P_h(z_0,z)\\g\neq 1\\g_1,g_2,g_3,g_4\mid g}}\prod_{p\mid g}\frac{4^364}{p^2}=\sum_{\substack{g\mid P_h(z_0,z)\\g\neq 1}}\sigma_0(g)^4\prod_{p\mid g}\frac{4^364}{p^2}\ll_\epsilon\sum_{g\geq z_0}g^{\epsilon-2}\ll_\epsilon z_0^{\epsilon-1}
     \end{multline*}
    Thus we have the following bound
    \begin{multline*}
        \sum_{\boldsymbol{\ell}\mid p_{h}(z_0,z)^4\atop \exists(j,m):\gcd(\ell_j,\ell_k)\neq 1}\Lambda_{\ell_1}^-\lambda_{\ell_2}^+\lambda_{\ell_3}^+\lambda_{\ell_4}^+(\prod_{p\mid\boldsymbol{\ell}}\beta_{X^{\boldsymbol{\ell}},p}(h)-\prod_{p\mid \ell_1}\beta_{X^{(\ell_1,1,1,1)},p}(h)\prod_{p\mid \ell_2}\beta_{X^{(1,\ell_1,1,1)},p}(h)
        \\\cdot\prod_{p\mid \ell_3}\beta_{X^{(1,1,\ell_1,1)},p}(h)\prod_{p\mid \ell_4}\beta_{X^{(1,1,1,\ell_1)},p}(h)).
        \\\ll_\epsilon z_0^{\epsilon-1}(\frac{\log(z)}{\log(z_0)})^{16}
    \end{multline*}
    Plugging in back to (6.3), we have 
    \begin{equation}
        \sum_{\ell_{1}\mid P_h(z_0, z)}\Lambda_{\ell_1}^-\sum_{\ell_{2}\mid P_h(z_0, z)}\lambda_{\ell_2}^+\sum_{\ell_{3}\mid P_h(z_0, z)}\lambda_{\ell_3}^+\sum_{\ell_{4}\mid P_h(z_0, z)}\lambda_{\ell_4}^+\prod_{p\mid\boldsymbol{\ell}}\beta_{X^{\boldsymbol{\ell}},p}(h)
    \end{equation}
    \begin{multline*}
        =\sum_{\ell_1\mid P_{h}(z_0,z)}\Lambda_{\ell_1}^-\prod_{p\mid \ell_1}\beta_{X^{(\ell_1,1,1,1)},p}(h)\sum_{\ell_2\mid P_{h}(z_0,z)}\lambda_{\ell_2}^+\prod_{p\mid \ell_2}\beta_{X^{(1,\ell_1,1,1)},p}(h)
        \\\cdot\sum_{\ell_3\mid P_{h}(z_0,z)}\lambda_{\ell_3}^+\prod_{p\mid \ell_3}\beta_{X^{(1,1,\ell_1,1)},p}(h)\sum_{\ell_4\mid P_{h}(z_0,z)}\lambda_{\ell_4}^+\prod_{p\mid \ell_4}\beta_{X^{(1,1,1,\ell_1)},p}(h)
        \\+O_{\epsilon}(z_0^{\epsilon-1}log(z)^{16})
    \end{multline*}
    By [3, (6.32)], we have
    \begin{equation}
        \sum_{\ell_2\mid P_{h}(z_0,z)}\lambda_{\ell_2}^+\prod_{p\mid \ell_2}\beta_{X^{(1,\ell_1,1,1)},p}(h)\geq \prod_{z_0< p\leq z\atop p\mid \prod_{j=1}^4\alpha_j}(1-\beta_{X^{(1,p,1,1)},p}(h))
    \end{equation}
    We then bound (for $z_0\geq5$)
    $$\prod_{z_0< p\leq z\atop p\mid \prod_{j=1}^4\alpha_j}(1-\beta_{X^{(1,p,1,1)},p}(h))\geq \prod_{z_0<p\leq z}(1-\frac{4}{p})\geq \frac{1}{\zeta(4)}\prod_{z_0<p\leq z}(1-\frac{1}{p})^4.$$
    Using [6, (3.30) and (3.26)], for $z_0$ sufficiently large we obtain from (6.8) that
    \begin{equation}
        \sum_{\ell_2\mid P_{h}(z_0,z)}\lambda_{\ell_2}^+\prod_{p\mid \ell_2}\beta_{X^{(1,\ell_1,1,1)},p}(h)\gg (\frac{\log(z_0)}{\log(z)})^4
    \end{equation}
    Taking $\kappa=4$ and 
    $$K:=\zeta(4)(1+\frac{1}{2\log(z)^2})^4(1+\frac{1}{\log(z_0)^2})^4$$
    We can conclude that for sufficiently large $z_0$
    \begin{equation}
        \prod_{z_0< p\leq z\atop p\mid \prod_{j=1}^4\alpha_j}(1-\beta_{X^{(1,p,1,1)},p}(h))\geq K^{-1}(\frac{\log(z_0)}{\log(z)})^4
    \end{equation}
    By choosing $z_0$ sufficiently large, we can take $K$ arbitrarily close to $\zeta(4)$ from above, say $K<1.645.$ By [3, Theorem 6.1], we hence obtain
    \begin{equation}
        \sum_{\ell_1\mid P_{h}(z_0,z)}\Lambda_{\ell_1}^-\prod_{p\mid \ell_1}\beta_{X^{(\ell_1,1,1,1)},p}(h)>(1-e^{37-s}(1.083^{10}))\prod_{z_0<p\leq z\atop p\mid\prod\alpha_j}\beta_{X^{(p,1,1,1)},p}(h).
    \end{equation}
    One verifies that for $s \geq 38$ we have
    $$1-e^{37-s}(1.083^{10})>0.$$
    Thus we have 
    \begin{multline*}
        \sum_{\ell_{1}\mid P_h(z_0, z)}\Lambda_{\ell_1}^-\sum_{\ell_{2}\mid P_h(z_0, z)}\lambda_{\ell_2}^+\sum_{\ell_{3}\mid P_h(z_0, z)}\lambda_{\ell_3}^+\sum_{\ell_{4}\mid P_h(z_0, z)}\lambda_{\ell_4}^+\prod_{p\mid\boldsymbol{\ell}}\beta_{X^{\boldsymbol{\ell}},p}(h)
        \\\geq (1.083)^{-4}(1-e^{37-s}(1.083^{10}))(\frac{\log(z_0)}{\log(z)})^{16}+O_{\epsilon}(z_0^{\epsilon-1}\log(z)^{16})
    \end{multline*}
    Taking $z_0=\log(z)^{33},$ we have
    \begin{equation}
      \sum_{\ell_{1}\mid P_h(z_0, z)}\Lambda_{\ell_1}^-\sum_{\ell_{2}\mid P_h(z_0, z)}\lambda_{\ell_2}^+\sum_{\ell_{3}\mid P_h(z_0, z)}\lambda_{\ell_3}^+\sum_{\ell_{4}\mid P_h(z_0, z)}\lambda_{\ell_4}^+\prod_{p\mid\boldsymbol{\ell}}\beta_{X^{\boldsymbol{\ell}},p}(h)
      \gg (\frac{\log(z_0)}{\log(z)})^{16}.
    \end{equation}
    Choose suitable $c_p,$ and again using [6, (3.30)], we conclude that
    \begin{equation}
        W_{\textbf{c},h}(z_0)\gg\frac{1}{\log(z_0)^3}.
    \end{equation}
    Thus the main term has a lower bound 
    \begin{equation}
        \gg \frac{\log(z_0)^{13}}{\log(z)^{16}}r_{gen^+(X^{\textbf{1}})}(h).
    \end{equation}
    Thus we conclude
    \begin{multline}
        S_{\textbf{c},h}(A_{\textbf{1}},z)\gg (\frac{\log(z_0)^{13}}{\log(z)^{16}}+O((\Delta^{-\frac{1}{2}}H(n)^5\log(D_0)^8
        \\+\Delta^{\epsilon-B}\log(z_0)^8+\Delta^Ce^{-s_0})(\frac{\log(z)}{\log(z_0)})^{6}))r_{gen^+(X^{\boldsymbol{\ell}})}(h)
        \\+O_{\epsilon}((DD_0)^{26+\epsilon}(\prod_{j=1}^4\alpha_j)^{\frac{11}{4}+\epsilon}h^{\frac{1}{2}+\epsilon}.
    \end{multline}
    Write $z=h^{\theta},s=38, D_0=h^{\epsilon},$ so $DD_0\ll_{\epsilon}h^{38\theta+\epsilon}$ and $s_0\gg_{\epsilon}\frac{\log(h)}{\log(\log(h))}.$ We can choose $\Delta=H(n)^{10}\log(h)^{\max(62,\frac{23}{B})}$.
    We then note that $\Delta\gg H(n)^{10}\log(h)^{62}$ together with $\log(D_0)\ll \log(h))$ implies that
    $$\Delta^{-\frac{1}{2}}H(n)^5\log(D_0)^{8}(\frac{\log(z)}{\log(z_0)})^{6}\ll_{\epsilon} \log(h)^{-17+\epsilon}\log(z_0)^{-6}$$
    while $\Delta\gg \log(h)^{\frac{23}{B}}$ implies
    $$\Delta^{\epsilon-B}\log(z_0)^8(\frac{\log(z)}{\log(z_0)})^{6}\ll \log(h)^{-17+\epsilon}\log(z_0)^{2}$$
    We can conclude that 
    \begin{equation}
        S_{\textbf{c},h}(A_{\textbf{1}},z)\gg \frac{\log(z_0)^{13}}{\log(z)^{16}}r_{gen^+(X^{\boldsymbol{\ell}})}(h)+O_{\epsilon}(h^{988\theta+\frac{1}{2}+\epsilon}(\prod_{j=1}^4\alpha_j)^{\frac{11}{4}+\epsilon})
    \end{equation}
    Then by Lemma 3.6 we have
    \begin{equation}
        \gg_{\epsilon}(\prod_{j=1}^4\alpha_j)^{-\frac{1}{2}}h^{1-\epsilon}+O(h^{988\theta+\frac{1}{2}+\epsilon}(\prod_{j=1}^4\alpha_j)^{\frac{11}{4}+\epsilon})
    \end{equation}
    Taking $$\theta<\frac{1}{1977}$$
    We have $$988\theta+\frac{1}{2}<1,$$
    and for sufficiently large $h$ we have
    $$S_{\textbf{c},h}(A_{\textbf{1}},z)>0$$
    Hence in (1.1) we may take that odd $p \mid x$ implies that
    $$p>h^{\theta}=(8(m-2)n+\sum_{j=1}^4\alpha_j(m-4)^2)^{\theta}\gg_{m}n^{\theta}=n^{\frac{1}{1977}}$$
    and $\text{ord}_p(x_j)<c_p$ for all $1\leq j\leq 4$
    If
    \begin{equation}
        \sum_{p>n^{\frac{1}{1900}}}\text{ord}_p(x)=a>0,
    \end{equation}
    then 
    $$p_m(x)\geq \frac{(m-2)x^2}{2}-\frac{\mid m-4\mid\cdot \mid x \mid}{2}\geq \frac{m-\frac{5}{2}}{2}x^2,$$
     Then, letting $\beta$ be the maximum of the respective choices of $a$ for $x_1,x_2,x_3,x_4$ from (6.18), we have
     $$n=\sum_{j=1}^4\alpha_jp_m(x_j)\gg_mn^{\frac{2\beta}{1977}},$$
     and we conclude that 
     $$\beta\leq 988.$$
      Overall, we obtain that $x_1,x_2,x_3,x_4$ are divisible by at most 
      $988$ primes.
      %$988+\sum_{p\mid \prod_{\alpha_j}}c_p$ primes.
    \end{proof}

    \begin{theorem}
        Let $h=8(m-2)n+\sum_{j=1}^4\alpha_j(m-4)^2$ and $\epsilon>0.$ Then $r_{X}(h)\neq 0$ if 
        \begin{equation*}
            h\gg_{\epsilon} (2(m-2))^{21+\epsilon}(\prod_{j=1}^4\alpha_j)^{6+\epsilon}.
        \end{equation*}
        \begin{proof}
            By Lemma 4.2, we have 
            \begin{align*}
                \mid a_{G_{X}}(h)\mid & \ll M^{\frac{11}{2}+\epsilon}N_{\boldsymbol{\alpha}}^{\frac{5}{2}+\epsilon}h^{\frac{1}{2}+\epsilon}
                \\& =(2(m-2))^{\frac{11}{2}}(4(m-2)^2\prod_{j=1}^4\alpha_j)^{\frac{5}{2}}h^{\frac{1}{2}+\epsilon}.  
            \end{align*}
            Then by comparing with the main term in (6.18), $r_{X}(h)\neq 0$ implies
            \begin{align*}
                (\prod_{j=1}^4\alpha_j)^{-\frac{1}{2}}h^{1-\epsilon}\gg_{\epsilon}(2(m-2))^{\frac{11}{2}}(4(m-2)^2\prod_{j=1}^4\alpha_j)^{\frac{5}{2}}h^{\frac{1}{2}+\epsilon}
            \end{align*}
            Therefore we have 
            $$h\gg_{\epsilon}(2(m-2))^{21+\epsilon}(\prod_{j=1}^4\alpha_j)^{6+\epsilon}.$$
        \end{proof}
    \end{theorem}

\def\refname{R\MakeLowercase{eferences}}

\end{document}